\documentclass[reqno,twoside,english,11pt]{article}

\usepackage{amsmath,amsfonts,calrsfs,fullpage,amssymb,xcolor,verbatim,eucal,yfonts,mathrsfs}

\usepackage[english,french]{babel}

\usepackage{mathtools}

\usepackage{color}
\usepackage{babel}
\usepackage{mathrsfs}
\usepackage{amsthm}
\usepackage[unicode=true,bookmarks=true,bookmarksnumbered=true,bookmarksopen=false,breaklinks=false,pdfborder={0 0 1},backref=false,colorlinks=true]{hyperref}
\hypersetup{
 pdfauthor={G Xi}}

\newtheorem{Theorem}{Theorem}[section]
\newtheorem{Definition}[Theorem]{Definition}
\newtheorem{Proposition}[Theorem]{Proposition}
\newtheorem{Lemma}[Theorem]{Lemma}

\newtheorem{Remark}[Theorem]{Remark}

\newtheorem{Hypothesis}{Hypothesis}

\numberwithin{equation}{section}
\theoremstyle{plain}

\theoremstyle{plain}

\theoremstyle{plain}

\theoremstyle{remark}

\theoremstyle{plain}

\theoremstyle{remark}
\newtheorem*{rem*}{\protect\remarkname}

\def\r2{\mathbb{R}^2}
\def\le{\left}
\def\r{\right}

\def\a{\alpha}
\def\d{\delta}

\def\e{\epsilon}
\def\la{\lambda}
\def\si{\sigma}

\def\B{{\cal B}}
\def\A{{\mathcal A}}

\def\H{{\mathcal H}}

\def\s0t{\sup_{t \in [0,T]}}

\def\ds{\displaystyle}

\def\beq{\begin{equation}}
\def\eeq{\end{equation}}
\def\barr{\begin{array}}
\def\earr{\end{array}}
\def\vs{\vspace{.01mm}   \\}
\def\rd{\reals\,^{d}}

\newcommand{\E}{\mathbb E}

\newcommand{\hsp}{\hspace{2truecm}}
\newcommand{\hslp}{\hspace{1truecm}}

\newcommand{\norm}[1]{\left\| #1\right\|}

\newcommand{\inner}[1]{\langle #1\rangle}
\newcommand{\Inner}[1]{\big\langle #1\big\rangle}

\newcommand{\abs}[1]{\lvert #1\rvert}

\def\H{{\mathcal{H}}}

\def\E{{\mathbb{E}}}
\def\P{{\mathbb{P}}}

\usepackage{babel}
\allowdisplaybreaks[4]

\date{}

  \providecommand{\corollaryname}{Corollary}
  
  \providecommand{\lemmaname}{Lemma}
\providecommand{\theoremname}{Theorem}

\theoremstyle{plain}

\makeatother

\providecommand{\corollaryname}{Corollary}
\providecommand{\lemmaname}{Lemma}
\providecommand{\propositionname}{Proposition}
\providecommand{\remarkname}{Remark}
\providecommand{\theoremname}{Theorem}

\begin{document}
\global\long\def\divg{{\rm div}\,}%

\global\long\def\curl{{\rm curl}\,}%

\global\long\def\rt{\mathbb{R}^{3}}%

\global\long\def\rd{\mathbb{R}^{d}}%

\global\long\def\rtwo{\mathbb{R}^{2}}%

\global\long\def\e{\epsilon}%

\title{On the small-mass limit for stationary solutions of stochastic wave equations with state dependent friction}

\author{Sandra Cerrai\thanks{Department of Mathematics, University of Maryland, cerrai@umd.edu. Partially supported by the NSF grant  DMS-1954299 - {\em Multiscale analysis of infinite-dimensional stochastic systems}.}, \ Mengzi Xie\thanks{Department of Mathematics, University of Maryland, mxie2019@umd.edu.}}

\maketitle

\selectlanguage{english}

\date{}

\begin{abstract}
We investigate the convergence, in the small mass limit, of the stationary solutions of a class of stochastic damped wave equations, where the friction coefficient depends on the state and the noisy perturbation if of multiplicative type. We show that the Smoluchowski-Kramers approximation that has been previously shown to be true in any fixed time interval, is still valid  in the long time regime. Namely we prove that the first marginals of any sequence of stationary solutions for the damped wave equation converge to the unique invariant measure of the limiting stochastic quasilinear parabolic equation. The convergence is proved with respect to the Wasserstein distance associated with the $H^{-1}$ norm.

\end{abstract}

\section{Introduction}
\label{introduction}

In this article we deal with the following   stochastic wave equation with state-dependent damping, on a bounded smooth domain $\mathcal{O}\subset \mathbb{R}^d$, with $d\geq 1$,
\begin{equation}
\label{SPDE1}
\le\{\begin{array}{l}
\ds{\mu\partial_t^2 u_\mu(t,x)=\Delta u_\mu(t,x)- \gamma (u_\mu(t,x)) \partial_t u_\mu(t,x) + f(x,u_{\mu}(t,x))+ \sigma(u_{\mu}(t,\cdot))\partial_t w^Q(t,x),}\\[10pt]
\ds{u_\mu(0,x)=\mathfrak{u}_0(x),\ \ \ \ \partial_t u_\mu(0,x)=\mathfrak{v}_0(x),\ \ \ \ \ \ \  u_{\mu}(t,x)=0,\ \ x \in\,\partial \mathcal{O},}
\end{array}\r.
\end{equation}
depending on a parameter $0<\mu<<1$. The friction coefficient $\gamma$ is a strictly positive, bounded and continuously differentiable function. The diffusion coefficient $\sigma$ is bounded and Lipschitz-continuous and the noise $w^Q(t)$ is a cylindrical $Q$-Wiener process, white in time and colored in space.  The nonlinearity $f:\mathcal{O}\times \mathbb{R}\to \mathbb{R}$ is Lipschtz-continuous with respect to the second variable and the identically zero function is globally asymptotically stable in the absence of the stochastic perturbation. Here and in what follows, we denote $H:=L^2(\mathcal{O})$, $H^{-1}:=H^{-1}(\mathcal{O})$, and $H^1:=H^1_0(\mathcal{O})$.

The solution $u_\mu(t,x)$ of equation \eqref{SPDE1} can be interpreted as the displacement  of the particles of a material continuum in a domain $\mathcal{O}$, subject to a random external force field $\partial_t w^Q(t,x)$ and a damping force which is proportional to the velocity field and depends on the state $u_\mu$. The second order differential operator takes into account of the interaction forces between neighboring particles, in  the presence of a non-linear reaction given by $f$. Here  $\mu$ represents the constant density of the particles and we are interested in the regime when $\mu\to 0$, known as the Smoluchowski-Kramers approximation limit (ref. \cite{kra} and \cite{smolu}).

In \cite{SK1} and \cite{SK2} it has been proven that, when $\gamma$ is constant, for every $T>0$ and $\eta>0$
\begin{equation}
\label{lim1-intro}
\lim_{\mu \to 0} \mathbb{P}\le(\sup_{t \in\,[0,T]}\Vert u_\mu(t)-u(t)\Vert_{H}>\eta \r)=0,	
\end{equation}
where $u \in\,L^2(\Omega; C([0,T];H)\cap L^2(0,T;H^1))$ is the solution of the parabolic problem
\begin{equation} \label{lim-par}
	\left\{\begin{array}{l}
\displaystyle{\gamma\partial_t u(t,x)= \Delta u(t,x)+f(x,u(t,x))+\sigma(u(t,\cdot))\partial_t w^Q(t,x),}\\[10pt]
		\displaystyle{u(0,x)=\mathfrak{u}_0(x), \ \ \ \ \ \ \ u(t,\cdot)\big|_{\partial\mathcal{O}}=0.}
	\end{array}\right.\end{equation}

When the friction coefficient $\gamma$ is state-dependent, the situation is more complicated and, because of the interplay between the noise and the non-constant friction, in the small-mass limit an extra drift term is created.  In this regard, in \cite{cerraixi} it has been proven that for every  $\mathfrak{u}_0 \in\,H^1$, $T>0$ and $p<\infty$, and for every $\eta>0$
\begin{equation} \label{lim_intro_p}  \lim_{\mu \to 0} \mathbb{P}\le(\int_0^T\Vert u_\mu(t)-u(t)\Vert^p_{H}\,dt>\eta \r)=0,\end{equation} where 
$u$ is the unique solution of the stochastic quasi-linear equation 
\begin{equation}\label{limiting_problem_intro}
	\left\{\begin{array}{l}
\displaystyle{\gamma(u(t,x))\partial_t u(t,x)= \Delta u(t,x)+f(x,u(t,x)) -\frac{\gamma'(u(t,x))}{2\gamma^2(u(t,x))} \sum_{i=1}^\infty |\sigma(u(t,\cdot))\,Qe_i(x)|^2 }\\[12pt]
\ds{\hsp \hsp+\sigma(u(t,\cdot))\partial_t w^Q(t,x)}\\[12pt]
		\displaystyle{u(0,x)=\mathfrak{u}_0(x), \ \ \ \ \ \ \ u(t,\cdot)\big|_{\partial\mathcal{O}}=0.}
	\end{array}\right.
\end{equation}

Notice that the case of a non-constant damping coefficient is not  the sole instance in which, within the context of a small mass limit, an additional drift term manifests. For example, in the case of a damped stochastic wave equation, constrained to live on the unitary sphere of $H$, in the limit the Smoluchowski-Kramers approximation yields a  stochastic parabolic problem also constrained to live on the unitary sphere of $H$, where an extra-drift term emerges, and that drift does not encompass the It\^o-to-Stratonovich correction (see \cite{BC23}). For a partial list of references where  this type of limit has been addressed in a variety of different contexts, see \cite{bhw}, \cite{SK3}, \cite{f}, \cite{fh}, \cite{hhv}, \cite{hmdvw}, \cite{hu}, \cite{lee}, \cite{spi}, for the finite dimension, and \cite{BC23}, \cite{SK1},  \cite{SK2},    \cite{cs}, \cite{cerraixi}, \cite{Lv2},  \cite{Nguyen}, \cite{salins}, for the 
infinite dimension.

After establishing the validity of the small mass limits within a fixed time interval $[0, T]$, the next step of interest is to compare the long-term dynamics of the second-order system with that of the first-order system (to this purpose, see e.g. \cite{cerraiglatt}, \cite{sal}, \cite{sal2}, \cite{CWZ}, \cite{cerraixie}, \cite{hu}). 

In \cite{SK1}, a comparative analysis of the long-term behavior of equations \eqref{SPDE1} (with a constant $\gamma$) and \eqref{lim-par} was conducted, assuming both systems to be of gradient type. Notably, in the case where the  noise is white in both space and time ($Q=I$) and the dimension is $d=1$, an explicit expression for the Boltzmann distribution of the process $z_\mu(t):=(u_{\mu}(t), \partial u_{\mu}/\partial t(t))$ in the phase space $\H:=L^2(0,L)\times H^{-1}(0,L)$ was derived. Since there is no equivalent of the Lebesgue measure in the functional space $\H$, an auxiliary Gaussian measure was introduced, and the density of the Boltzmann distribution was then expressed with respect to such auxiliary Gaussian measure, which itself corresponds to the stationary measure of the linear wave equation associated with problem \eqref{SPDE1}.
In particular, it was shown  that the first marginal of the invariant measure linked to the process $z_{\mu}(t)$ remains independent of $\mu>0$ and coincides with the invariant measure for the heat equation \eqref{lim-par}.

In the case of non-gradient systems, that is when the noise is colored in space and/or of multiplicative type,  there is no  explicit expression for the invariant measure $\nu_\mu$ associated with system \eqref{SPDE1} and there is no reason to expect that the first marginal of $\nu_\mu$ does not depend on $\mu$ or  coincides with the invariant measure $\nu$ of system \eqref{lim-par}. Nonetheless, in \cite{cerraiglatt} it was proved that, as the mass parameter $\mu$ tends to zero, the first marginal of any invariant measure $\nu_\mu$ associated with the second-order system \eqref{SPDE1} converges in a suitable manner to the invariant measure $\nu$ of the first-order system \eqref{lim-par}.
Specifically, the following convergence was established
\begin{equation}
\label{intro4}
\lim_{\mu\to 0} \mathcal{W}_\a\left((\Pi_1\nu_\mu)^\prime,\nu\right)=0,
\end{equation}
where $(\Pi_1\nu_\mu)^\prime$ denotes the extension of the first marginal of the invariant measure $\nu_\mu$ to $H$, and the metric $\mathcal{W}_\a$ corresponds to the Wasserstein metric on $\mathcal{P}(H)$ associated with a distance metric $\a$ in $H,$ which was determined based on the characteristics of the non-linearity function $f$ under consideration.

\medskip

In the present paper, we want to see if any of the results proved in \cite{cerraiglatt} in the case of a constant friction $\gamma$, can be proven for of a state-dependent $\gamma$, where the Smoluchowski-Kramers approximation gives the stochastic quasi-linear parabolic problem \eqref{limiting_problem_intro}, instead of the simpler parabolic semi-linear problem \eqref{lim-par}. 

One of the key ingredients used in  \cite{cerraiglatt} for the  proof of \eqref{intro4}  is the fact that  the transition semigroup $P_t$ associated with equation \eqref{lim-par} admits a unique invariant measure $\nu \in\,\mathcal{P}(H)$ and the following contraction property  holds
\begin{equation}
\label{intro5}
\mathcal{W}_\a\le(P_t^\star \nu_1,P_t^\star \nu_2\r)\leq c\,e^{-\delta t}\,\mathcal{W}_\a(\nu_1,\nu_2),\ \ \ \ t\geq 0,\ \ \ \ \ \nu_1, \nu_2 \in\,\mathcal{P}(H),
\end{equation}
for some $\delta>0$.
In the case of equation \eqref{lim-par}, this kind of problems have been studied extensively and a wide variety of results is available. 
 However, in the case of the quasi-linear problem \eqref{limiting_problem_intro} the situation is considerably more delicate and several fundamental facts are not known, as for one  whether the semigroup associated is Feller in $H$ or not. In particular, even the use of the Krylov-Bogoliubov theorem for the proof of the existence of an invariant measure in $H$ is not possible.
 Thus, in the present paper we have to follow a different path, that in particular brings us to study equations \eqref{SPDE1} and \eqref{limiting_problem_intro} in spaces of lower regularity than $H^1\times H$ and $H$, respectively.
  
 In \cite{cerraixi}, it has been proved that equation \eqref{SPDE1} is well-posed in $\mathcal{H}_1:=H^1\times H$, for every $\mu>0$, so that the associated Markov transition semigroup $P^{\mu, \mathcal{H}_1}_t$ can be introduced. Our first step is showing that in fact \eqref{SPDE1} is well-posed also in $\mathcal{H}:=H\times H^{-1}$, for every $\mu>0$,  and there exists an invariant measure $\nu^{\mu, \mathcal{H}}$ for the corresponding transition semigroup $P^{\mu, \mathcal{H}}_t$. We show that such invariant measure is supported in $\mathcal{H}_1$ and its restriction to $\mathcal{H}_1$ is invariant for $P^{\mu, \mathcal{H}_1}_t$. Moreover, we prove suitable uniform bounds for the moments of $\nu^{\mu, \mathcal{H}}$ and $\nu^{\mu, \mathcal{H}_1}$, which are fundamental for the proof of the limit.
 
 Next, we move our analysis to the limiting equation \eqref{limiting_problem_intro}. As already done in \cite{cerraixi} and \cite{cerraixie}, we do not work directly with \eqref{limiting_problem_intro}, but rather with its equivalent formulation 
 \begin{equation}\label{limit_para_intro}
	\le\{\begin{array}{l}
		\ds{\partial_{t}\rho(t,x)=\text{div}\left(\frac{1}{\gamma(g^{-1}(\rho(t,x)))}\nabla\rho(t,x)\right)+f(x,g^{-1}(\rho(t,x)))+\sigma(g^{-1}(\rho(t,\cdot)))\partial_{t}w^{Q}(t,x), }\\[16pt]
		\ds{\rho(0,x)=g(\mathfrak{u}_0(x)),\ \ \ \ \ \ \  \rho(t,\cdot)|_{\partial\mathcal{O}}=0 },
	\end{array}\r.
\end{equation}
where 
$g$ is the antiderivative of $\gamma$ vanishing at zero. Since we are assuming that $\gamma$ is strictly positive, bounded and continuously differentiable, the mappings 
\[h \in\,H\mapsto g\circ h \in\,H,\ \ \ \ \ \ \ h \in\,H^1\mapsto g\circ h \in\,H^1\] are both  homeomorphisms and the coefficients in \eqref{limit_para_intro} are all well defined and regular.   Moreover, as shown in \cite{cerraixi}, by using a generalized It\^o's formula,  for every $\mathfrak{r}_0=g(\mathfrak{u}_0) \in\,H^1$ and $t\geq 0$   we have that 
\[\rho^{\mathfrak{r}_0}(t)=g(u^{\mathfrak{u}_0}(t)),\ \ \ \ \ \ \ \ \ \ \ g^{-1}(\rho^{\mathfrak{r}_0}(t))=u^{\mathfrak{u}_0}(t).\]
In particular,  equation \eqref{limit_para_intro} is well posed  in $C([0,T];H)\cap L^2(0,T;H^1)$ if and only if equation \eqref{limiting_problem_intro} is.

As a consequence of \eqref{lim_intro_p}, we have that for every initial condition $\mathfrak{r}_0 \in\,H^1$ equation \eqref{limit_para_intro} has a unique solution $\rho^{\mathfrak{r}_0} \in\,L^2(\Omega;L^p(0,T;H^1))$, with $p<\infty$. However, since 
the long time behavior of  \eqref{limit_para_intro} in $H^1$ and $H$ is not well understood,  we need to study its well-posedness in $H$ and   $H^{-1}$, so that we can introduce the corresponding transition semigroups $R^{H}_t$ and $R^{H^{-1}}_t$.  Due to the equivalence of problems \eqref{limiting_problem_intro} and \eqref{limit_para_intro} in $H$ this allows to introduce the transition semigroup $P^H_t$ associated with equation \eqref{limiting_problem_intro}. 

Next, we  prove that there exists some constant $\lambda>0$ such that for every $\mathfrak{r}_1, \mathfrak{r}_2 \in\,H^{-1}$ and $t\geq 0$  \begin{equation}  \label{contraction_intro}
		\E\norm{\rho^{\mathfrak{r}_1}(t)-\rho^{\mathfrak{r}_1}(t)}_{H^{-1}}^{2}\leq e^{-\lambda t}\norm{\mathfrak{r}_1-\mathfrak{r}_2}_{H^{-1}}^{2},\ \ \ \ \ \ t\geq 0.
	\end{equation}
To this purpose, we would like to mention that in \cite{fridnew} and \cite{GH}, it was proved that
under suitable conditions on the initial conditions, the following property holds
\[\E\norm{\rho^{\mathfrak{r}_1}(t)-\rho^{\mathfrak{r}_1}(t)}_{L^1(\mathcal{O})}^{2}\leq \norm{\mathfrak{r}_1-\mathfrak{r}_2}_{L^1(\mathcal{O})}^{2},\ \ \ \ \ \ t\geq 0.\] Such bound  gives in particular the Feller property in $L^1(\mathcal{O})$ but, unfortunately, this is not useful to our analysis, as it is not clear how to handle the proof of our limiting problem in a $L^1(\mathcal{O})$ setting. As far as we know, it is not clear if such a bound is satisfied in $H$, and, as we already mentioned above,  this is why we need to work in $H^{-1}$, where we have the validity of \eqref{contraction_intro}.

As a consequence of  \eqref{contraction_intro}, we have  that 
$R^{H^{-1}}_t$ is Feller. This, together with suitable uniform bounds in $H^1$, allows to conclude that $R^{H^{-1}}_t$ has an invariant measure $\nu^{H^{-1}}$, supported in $H^1$. Moreover \eqref{contraction_intro} implies 
 that for every $\varphi\in\text{Lip}_{b}(H^{-1})$ and $\mathfrak{r}_1, \mathfrak{r}_2\in H^{-1}$ 
	\begin{equation}\label{sm9_intro}
		\left\lvert R_t^{H^{-1}}\varphi(\mathfrak{r}_1)-R_t^{H^{-1}}\varphi(\mathfrak{r}_2)\right\rvert\leq [\varphi]_{\text{Lip}_{H^{-1},\alpha}}e^{-\lambda t/2}\norm{\mathfrak{r}_1-\mathfrak{r}_2}_{H^{-1}},\ \ \ \ \ \ \ \ \   t\geq 0,
	\end{equation}
	so that $\nu^{H^{-1}}$ is the unique invariant measure of $R^{H^{-1}}_t$, and $\nu^H$, its restriction to $H$, turns out to be the unique invariant measure of $R^H_t$. Finally, due to the {\em equivalence} between equations \eqref{limiting_problem_intro} and \eqref{limit_para_intro}, we show that this implies that $P^{H}_t$ has a unique invariant measure $\nu$. 
	
	By using a general argument developed in \cite{FoldesGlattHoltzRichardsThomann2013}, and already used in \cite{cerraiglatt} in a similar context, all this allows to obtain our main result. Namely, we can show that  if we define
\[\alpha(\mathfrak{u}_1,\mathfrak{u}_2):=\norm{\mathfrak{u}_1-\mathfrak{u}_2}_{H^{-1}},\ \ \ \ \ \ \mathfrak{u}_1, \mathfrak{u}_2\in H^{-1},\]
then we have
	\begin{equation} \label{fine?}
\lim_{\mu\to 0} \mathcal{W}_\alpha\left(\Pi_1 \nu_\mu^{\mathcal{H}},		\nu\right)=0.
	\end{equation}
	Actually, due to \eqref{sm9_intro} and the invariance of 
	 $\nu_\mu^{\mathcal{H}}$ and $\nu^{H^{-1}}$  we have
\[\begin{aligned}
\mathcal{W}_{\alpha}&\left(\left[\left(\Pi_{1}\nu_{\mu}^{ \mathcal{H}}\right)\circ g^{-1}\right]',\nu^{H^{-1}}\right)\leq \mathcal{W}_{\alpha}\left(\left[\Pi_1((P_t^{\mu, \mathcal{H}})^\star \nu_\mu^{\mathcal{H}})\circ g^{-1}\right]^\prime,(R^{H^{-1}}_{t})^{\ast}\left[\left(\Pi_{1}\nu_{\mu}^{ \mathcal{H}}\right)\circ g^{-1}\right]'\r)\\[10pt]
&\hsp \hsp\hsp+c\, e^{-\lambda t}\,\mathcal{W}_{\alpha}\left(
\left[\left(\Pi_{1}\nu_{\mu}^{ \mathcal{H}}\right)\circ g^{-1}\right]',\nu^{H^{-1}}
\right),	
\end{aligned}
\]
and then, if we pick $\bar{t}>0$ such that $c e^{-\lambda \bar{t}}\leq 1/2$, we obtain
\[\mathcal{W}_{\alpha}\left(\left[\left(\Pi_{1}\nu_{\mu}^{ \mathcal{H}}\right)\circ g^{-1}\right]',\nu^{H^{-1}}\right)\leq 2\,\mathcal{W}_{\alpha}\le(\left[(\Pi_1((P_{\bar{t}}^{\mu, \mathcal{H}})^\star \nu_\mu^{\mathcal{H}})\circ g^{-1}\right]^\prime,(R^{H^{-1}}_{\bar{t}})^{\ast}\left[\left(\Pi_{1}\nu_{\mu}^{ \mathcal{H}}\right)\circ g^{-1}\right]'\r)\]
(here we are using the notation $[\cdot]^\prime$ to denote the extension to $H^{-1}$ of  an arbitrary probability measure defined in $H$).
Thanks to the  Kantorovich-Rubinstein duality, we have
\[\mathcal{W}_{\alpha}\le(\left[(\Pi_1((P_t^{\mu, \mathcal{H}})^\star \nu_\mu^{\mathcal{H}})\circ g^{-1}\right]^\prime,(R^{H^{-1}}_{t})^{\ast}\left[\left(\Pi_{1}\nu_{\mu}^{ \mathcal{H}}\right)\circ g^{-1}\right]'\r)\leq \mathbb{E}\,\alpha (g(u_\mu^{\zeta_\mu}(t)), \rho^{\,g(\xi_\mu)}(t)),\] 
for every  $\mathcal{F}_{0}$-measurable $\mathcal{H}_{1}$-valued random variable $\zeta_\mu:=(\xi_\mu,\eta_\mu)$, distributed as the invariant measure $\nu_{\mu}^{ \mathcal{H}}$.
Hence,    once we prove that for every $t\geq 0$ large enough
	\[
\lim_{\mu\to 0}\,\mathbb{E}\,\alpha (g(u_\mu^{\zeta_\mu}(t)), \rho^{g(\xi_\mu)}(t))=\lim_{\mu\to 0}\E\,\Vert g(u_\mu^{\zeta_\mu}(t))-\rho^{g(\xi_\mu)}(t)\Vert_{H^{-1}}=0,
	\]
we obtain that 
\[\lim_{\mu\to0}\mathcal{W}_{\alpha}\left(\left[\left(\Pi_{1}\nu_{\mu}^{ \mathcal{H}}\right)\circ g^{-1}\right]',\nu^{H^{-1}}\right)=0,\]
and we conclude by showing that this implies \eqref{fine?}. Moreover, from the proof of \eqref{fine?} we also get that  $\Pi_1 \nu_\mu^{\mathcal{H}}$ converges to $\nu$, weakly in $H$, as $\mu\downarrow 0$.

%%%%%%%%%%%%%%%%%%%%%%%%%%%%%%%%%%%%%%%%%%%%%%%%%%

\section{Notations and assumptions}\label{assumption}

Throughout the present paper $\mathcal{O}$ is a bounded domain in $\mathbb{R}^d$, with $d\geq 1$, having a smooth boundary. We denote by $H$ the Hilbert space $L^2(\mathcal{O})$ and by $\Vert\cdot\Vert_H$ and $\langle \cdot , \cdot \rangle_H$ the corresponding  norm and inner product.  

Given the domain $\mathcal{O}$, we denote by $A$ the realization of the Laplace operator $\Delta$, endowed with Dirichlet boundary conditions. As known there exists a complete orthonormal basis
$\{e_i\}_{i\in\mathbb{N}}$  of $H$ which diagonalizes $A$. In what follows,  we denote by $\{-\alpha_i\}_{i\in \mathbb{N}}$ the corresponding sequence of eigenvalues, and for every $\d \in\,\mathbb{R}$, we define  $H^\delta$ as the completion of $C^{\infty}_0(\mathcal{O})$ with respect to the norm
\[\Vert u\Vert_{H^\delta}^2:=\sum_{i=1}^\infty \alpha_i^\delta \langle h,e_i\rangle_H^2.\]
Notice that with this definition $H^0=H$ and, if $\delta_1<\delta_2$, then $H^{\delta_2}\hookrightarrow H^{\delta_1}$ with compact embedding.
 We also define 
\[\mathcal{H}_\delta:=H^\delta\times H^{\delta-1},\ \ \ \ \ \ \ \  \mathcal{H}:=H\times H^{-1}.\]

Next, for every two separable Hilbert spaces $E$ and $F$, we denote by $\mathcal{L}_2(E,F)$  the space of Hilbert-Schmidt operators from $E$ into $F$. $\mathcal{L}_2(E,F)$ is a Hilbert space, endowed with the inner product
\[\langle B, C\rangle_{\mathcal{L}_2(E,F)}=\mbox{Tr}_E\,[B^\star C]=\mbox{Tr}_F[C B^\star],\]
and, as well known, $\mathcal{L}_2(E,F) \subset \mathcal{L}(E,F)$, with
\[
	\Vert B\Vert_{\mathcal{L}(E,F)}\leq 	\Vert B\Vert_{\mathcal{L}_2(E,F)}.
\]

Finally, if $X$ is any Polish space, we denote by $B_b(X)$ the space of Borel bounded functions $\varphi:X\to \mathbb{R}$, endowed with the sup-norm
\[\Vert \varphi\Vert_{\infty}:=\sup_{h \in\,X}\vert \varphi(h)\vert.\]
 Moreover, we denote by $C_b(X)$ the subspace of uniformly continuous and bounded functions.

%%%%%%%%%%%%%%%%%%%%%%%%%%%%%%%%%%%%%%%%%%

\subsection{Assumptions}
We assume that $w^Q(t)$ is a cylindrical $Q$-Wiener process, defined on a complete stochastic basis  $(\Omega,\mathcal{F},(\mathcal{F}_t)_{t\geq 0},\mathbb{P})$. This means that $w^Q(t)$ can be formally written as 
\[
w^Q(t)=\sum_{i=1}^\infty Q e_i \beta_i(t),
\]
where $\{\beta_i\}_{i\in \mathbb{N}}$ is a sequence of independent standard Brownian motions on $(\Omega,\mathcal{F},(\mathcal{F}_t)_{t\geq 0},\mathbb{P})$,   $\{e_i\}_{i\in\,\mathbb{N}}$ is the complete orthonormal system introduced above that diagonalizes the Laplace operator, endowed with Dirichlet boundary conditions, and $Q:H\rightarrow H$ is a bounded  linear operator. 
When $Q=I$, the process $w^I(t)$ will be denoted by $w(t)$. In particular, we have $w^Q(t)=Qw(t)$.

In what follows we shall denote by $H_Q$ the set $Q(H)$. $H_Q$ is the reproducing kernel of the noise $w^Q$ and  is a Hilbert space, endowed with the inner product
\[\langle h, k\rangle_{H_Q}=\langle Q^{-1}h,Q^{-1}k\rangle_H,\ \ \ \ h, k \in\,H_Q.\] Notice that the sequence $\{Q e_i\}_{i \in\,\mathbb{N}}$ is a complete orthonormal system in $H_Q$. Moreover, if $U$ is any Hilbert space containing $H_Q$ such that the embedding  of $H_Q$ into $U$ is Hilbert-Schmidt, we have that 
\[w^Q \in\,C([0,T];U).
\]

\begin{Hypothesis}
	\label{Hypothesis1}
	The mapping $\si:H\to \mathcal{L}_2(H_Q,H)$ is defined by 
	\[[\sigma(h)Qe_i](x) = \sigma_i(x,h(x)), \ \ \ \ x \in\,\mathcal{O}\ \ \ \ \ h \in\,H,\ \ \ \ i \in\,\mathbb{N},\] 
	for some measurable mappings $\sigma_i:\mathcal{O}\times \mathbb{R}\rightarrow \mathbb{R}$. We assume that there exists $L_\sigma>0$ such that
	\begin{equation}
		\label{sgfine1}	
		\sup_{x \in\,\mathcal{O}}\, \sum_{i=1}^\infty \vert \sigma_i(x,y_1) - \sigma_i(x,y_2)\vert^2 \leq L_\sigma\,\vert y_1-y_2\vert^2,\ \ \ \ \ y_1, y_2 \in\,\mathbb{R}.
	\end{equation}
	Moreover, we assume $\sigma$ is bounded, that is,
	\begin{equation}\label{sgfine2}
		\sigma_{\infty}:=\sup_{h\in H}\Vert\sigma(h)\Vert _{\mathcal{L}_{2}(H_{Q},H)}<\infty.
	\end{equation}
\end{Hypothesis}

\begin{Remark}
{\em \begin{enumerate}
\item[1.] Condition \eqref{sgfine1} implies that $\si:H\to \mathcal{L}_2(H_Q,H)$ is Lipschitz continuous. Namely,
			for any $h_1,h_2\in H$ 
			\begin{equation}   \label{sigma-lip}
				\Vert\si(h_1)-\si(h_2)\Vert_{\mathcal{L}_2(H_Q,H)} \leq \sqrt{L_\sigma}\, \Vert h_1 -h_2\Vert_H.
			\end{equation}
		\item[2.] If the noise is additive, then Hypothesis \ref{Hypothesis1} is satisfied when $\text{Tr}Q^{2}<+\infty$. 
	
\end{enumerate}}
\end{Remark}

\begin{Hypothesis}\label{Hypothesis2}
	The mapping $\gamma$ belongs to $C^1_b(\mathbb{R})$ and there exist $\gamma_0$ and $\gamma_1$ such that
	\begin{equation}
		\label{nonlinearity assumption}
		0<\gamma_0\leq \gamma( r)\leq \gamma_1,\ \ \ \ \ \  r\in\mathbb{R}.
	\end{equation}
	
\end{Hypothesis}

If we define
\[g(r):=\int_0^r \gamma(\si)\,d\si,\ \ \ \ \ r \in\,\mathbb{R},\]
the function $g:\mathbb{R}\to\mathbb{R}$ is differentiable, strictly increasing and invertible so that its inverse $g^{-1}:\mathbb{R}\to\mathbb{R}$ is differentiable, with
\begin{equation}
	\label{sm3}
	\sup_{r \in\,\mathbb{R}}\,(g^{-1})^\prime(r)\leq \frac 1{\gamma_0}.
\end{equation}

\begin{Hypothesis}
	\label{Hypothesis3}
	The mapping $f:\mathcal{O}\times\mathbb{R}\to\mathbb{R}$ is measurable and   there exists a positive constant $L_f$ such that
	\begin{equation}
		\label{L_F_condition}
		L_f<\frac{\alpha_{1}\gamma_{0}}{\gamma_{1}},
	\end{equation}
	and
	\begin{equation}\label{f_Lipchitz}
		\sup_{x \in\,\mathcal{O}}\abs{f(x,r)-f(x,s)}\leq L_f\abs{r-s},\ \ \ r,s\in\mathbb{R}.
	\end{equation}
Moreover,
	\begin{equation*}
		\sup_{x\in\mathcal{O}}\abs{f(x,0)}<\infty.
	\end{equation*}
\end{Hypothesis}

In what follows, for every $x\in\mathcal{O}$ and  $r\in\mathbb{R}$ we denote 
\begin{equation*}
	\mathfrak{f}(x,r):=\int_{0}^{r}f(x,s)ds,
\end{equation*}
and for every function $h:\mathcal{O}\to\mathbb{R}$, we  denote
\begin{equation*}
	F(h)(x):=f(x,h(x)),\ \ \ \ \ \ \ \ x\in\mathcal{O}.
\end{equation*}

\begin{Remark}
{\em \begin{enumerate}	
\item[1.]  Condition \eqref{f_Lipchitz} implies that $F:H\to H$ is Lipschitz continuous. Namely for any $h_{1},h_{2}\in H$
	\[
		\norm{F(h_{1})-F(h_{2})}_{H}\leq L_f\norm{h_{1}-h_{2}}_{H}.
	\]
	Moreover, there exists $c>0$ such that
	\begin{equation}
	\label{sm1}
	\Vert F(h)\Vert_{H}\leq L_f\,\Vert h\Vert_H+c.	
	\end{equation}

\item[2.]  If the friction coefficient $\gamma$ is constant, then $\gamma_{0}=\gamma_{1}$, and condition \eqref{L_F_condition} becomes
	\begin{equation*}
		L_f<\alpha_{1}.
	\end{equation*}
	\item[3.] It is immediate to check that if for every $h \in\,H$ we define
	\[\Lambda(h):=\int_{\mathcal{O}}f(x,h(x))\,dx,\]
	then
	\begin{equation}
	\label{sm2}
	[D\Lambda(h)](x)=f(x,h(x)),\ \ \ \ \ x \in\,\mathcal{O}.	
	\end{equation}

	\end{enumerate}}
\end{Remark}

\begin{Hypothesis}\label{Hypothesis_control}
	We assume 
	\begin{equation}\label{additional_condition}
		L_f+\frac{L_\sigma}{2\gamma_{0}}<\frac{\alpha_{1}\gamma_{0}}{\gamma_{1}}.
\end{equation}
\end{Hypothesis}

\begin{Remark}
{\em If $\sigma$ is constant, then $L_\sigma=0$ and Hypothesis \ref{Hypothesis_control} reduces to condition \eqref{L_F_condition} in Hypothesis 	\ref{Hypothesis3}.}
\end{Remark}

%%%%%%%%%%%%%%%%%%%%%%%%%%%%%%%%%%%%%%%%%%%%%%%%%%%%

\section{The main result}
\label{sec3}

For every $\mu>0$, we denote $v_{\mu}:=\partial_{t}u_{\mu}$, and rewrite equation \eqref{SPDE1} as the following system
\begin{equation}\label{system}
	\le\{\begin{array}{l}
		\ds{du_{\mu}(t)=v_{\mu}(t)dt, }\\[10pt]
		\ds{\mu dv_{\mu}(t)=\big[ A u_{\mu}(t)-\gamma(u_{\mu}(t))v_{\mu}(t)+F(u_{\mu}(t))\big]dt+\sigma(u_{\mu}(t))dw^{Q}(t), }\\[10pt]
		\ds{u_\mu(0)=\mathfrak{u}_0,\ \ \ \ \ \ \ \ \ v_\mu(0)=\mathfrak{v}_0, }
	\end{array}\r.
\end{equation}
where $A$ is the realization  in $H$ of the Laplacian $\Delta$,  endowed with Dirichlet boundary conditions. If we define 
\begin{equation*}
	\eta:=\frac{1}{\sqrt{\mu}}\big(\mu\partial_{t}u+g(u)\big),\ \ \ \zeta=(u,\eta),
\end{equation*}
system \eqref{system} can be rewritten as
\begin{equation}\label{abstract}
	d\zeta_{\mu}(t)=\mathcal{A}_{\mu}(\zeta_{\mu}(t))dt+\Sigma_{\mu}(\zeta_{\mu}(t))dw^{Q}(t),\ \ \ \zeta_{\mu}(0)=\Big(\mathfrak{u}_0,\sqrt{\mu} \mathfrak{v}_0+\frac{g(\mathfrak{u}_0)}{\sqrt{\mu}}\Big),
\end{equation}
where we denoted 
\begin{equation*}
	\mathcal{A}_{\mu}(\zeta):=\frac{1}{\sqrt{\mu}}\Big(\eta-\frac{g(u)}{\sqrt{\mu}}, A u+F(u)\Big),\ \ \ \zeta=(u,\eta)\in D(\mathcal{A}_{\mu})=\mathcal{H}_{1},
\end{equation*}
and 
\begin{equation*}
\Sigma_{\mu}(\zeta):=\frac{1}{\sqrt{\mu}}\big(0,\sigma(u)\big)	,\ \ \ \zeta=(u,\eta)\in\mathcal{H}.
\end{equation*}

This means that, for every $\mu>0$ and every $(\mathfrak{u}_0,\mathfrak{v}_0)\in \H_{1}$, the adapted $\H_{1}$-valued process $\zeta_{\mu}=(u_\mu, \eta_\mu)$ is the unique solution of equation \eqref{abstract}, with $\zeta_{\mu}(0)=\big(\mathfrak{u}_0,\sqrt{\mu}\mathfrak{v}_0+g(\mathfrak{u}_0)/\sqrt{\mu}\big)$, if and only if the adapted $\H_{1}$-valued process \[z_\mu(t):=(u_{\mu}(t),v_{\mu}(t))=(u_\mu(t), \eta_\mu(t)/\sqrt{\mu}-g(u_\mu(t))/\mu),\ \ \ \ \ \ t\geq 0,\] is the unique solution of system \eqref{system}, with $z_\mu(0)=\mathfrak{z}_0:=(\mathfrak{u}_0,\mathfrak{v}_0)$.

In \cite{cerraixi} it has been proven that, under Hypotheses \ref{Hypothesis1}, \ref{Hypothesis2} and \ref{Hypothesis3}, without condition \eqref{L_F_condition}, 
for every $\zeta^0\in L^{2}(\Omega;\mathcal{H}_{1})$ and for every $\mu, T>0$, there exists a unique solution 
$\zeta_{\mu}\in L^{2}(\Omega;C([0,T];\mathcal{H}_{1}))$ for equation \eqref{abstract}, with $\zeta_{\mu}(0)=\zeta^0$.
In particular, this implies that for every $(\mathfrak{u}_0,\mathfrak{v}_0)\in L^{2}(\Omega;\H_{1})$, and for every $\mu, T>0$, there exists a unique solution $z_\mu\in L^{2}(\Omega;C([0,T];\mathcal{H}_{1}))$ to equation \eqref{system} with $z_\mu(0)=(\mathfrak{u}_0,\mathfrak{v}_0)$.

In what follows, we shall denote by 
$P_{t}^{\mu, \mathcal{H}_1}$, the transition semigroup associated with equation \eqref{system} in $\mathcal{H}_1$, which is defined as
\begin{equation*}
	P_{t}^{\mu, \mathcal{H}_1}\varphi(\mathfrak{z})=\mathbb{E}\,\varphi\big(z_{\mu}^{\mathfrak{z}}(t)\big),\ \ \ \ \  \mathfrak{z}\in\mathcal{H},\ \ \ t\geq0,
\end{equation*} 
for every $\varphi\in B_{b}(\mathcal{H}_1)$.

%%%%%%%%%%%%%%%%%%%%%%%%%%%%%%%%%%%%%%%%%%%%%%%%%

\subsection{Generalized solutions of equations \eqref{system} and \eqref{abstract}}

We start with the definition of generalized solution for equation \eqref{abstract}.
\begin{Definition}
For every $\mu,T>0$ and every $\zeta^0\in \mathcal{H}$, we say that $\zeta_{\mu}\in L^{2}(\Omega;C([0,T];\mathcal{H}))$ is a generalized solution of   problem
	\eqref{abstract}, with initial condition $\zeta^0$, 
	 if there exists a sequence $\{\zeta^0_n\}_{n\in\mathbb{N}}\subset\H_{1}$ converging to $\zeta^0$ in $\H$, as $n\to+\infty$, such that
	 \begin{equation*}
	 	\lim_{n\to +\infty}\zeta_{\mu,n}=\zeta_{\mu} \ \ \text{in}\ \  L^{2}(\Omega;C([0,T];\mathcal{H})),
	 \end{equation*}
 	where $\zeta_{\mu,n}\in L^{2}(\Omega;C([0,T];\H_{1}))$ is the unique solution of equation \eqref{abstract} with initial condition $\zeta^0_n$. 
\end{Definition}

First, we study the existence and uniqueness of generalized solutions of  equation \eqref{abstract} in the space $\mathcal{H}$.

	\begin{Lemma}\label{wellposedness_gen_z}
		Under Hypotheses \ref{Hypothesis1}, \ref{Hypothesis2} and \ref{Hypothesis3}, for every $\mu,T>0$ and every $\zeta^0\in\mathcal{H}$, there exists a unique generalized solution $\zeta_{\mu}\in L^{2}(\Omega;C([0,T];\mathcal{H}))$ for equation \eqref{abstract}, and such solution coincides with the classical solution when $\zeta^0 \in\,\mathcal{H}_1$.  Moreover, if $\zeta_{\mu}^{1}, \zeta_{\mu}^{2}$ are two generalized solutions of \eqref{abstract}, with initial conditions $\zeta^1, \zeta^2\in\H$, respectively, then 
		\begin{equation}\label{Feller_z}
			\E\sup_{t \in\,[0,T]}\Vert \zeta_{\mu}^{1}(t)-\zeta_{\mu}^{2}(t)\Vert_{\H}^{2}\leq e^{c_{\mu}T}\Vert \zeta^1-\zeta_{ 2}\Vert_{\H}^{2},
		\end{equation}
		for some constant $c_{\mu}$.
	\end{Lemma}
	
	\begin{proof}  
	Without any loss of generality, we assume $\mu=1$, and for simplicity of notation, we  denote $\mathcal{A}_{1}$ and $\Sigma_1$ by $\mathcal{A}$ and $\Sigma$, respectively. In \cite{cerraixi}, it is proved that the operator $\A$ is quasi-$m$-dissipative in $\mathcal{H}$. Namely, 
	 there exists $\omega\geq0$ such that for every $\zeta,\theta\in D(\mathcal{A})$
		\begin{equation}
		\label{sm20}
			\Inner{\mathcal{A}(\zeta)-\mathcal{A}(\theta),\zeta-\theta}_{\mathcal{H}}\leq \omega\norm{\zeta-\theta}_{\mathcal{H}}^{2},
		\end{equation}
		and there exists $\lambda_{0}>0$ such that 
		\[
			\text{Range}(I-\lambda\mathcal{A})=\mathcal{H},\ \ \ \lambda\in(0,\lambda_{0}).
		\]

Now, let $\zeta^0\in \H$ and let $\{\zeta^0_n\}_{n\in\mathbb{N}}\subset \H_{1}$ be any sequence converging to $\zeta^0$ in $\H$. For every $n\in\mathbb{N}$, we denote by $\zeta_{n}$ the unique solution of equation \eqref{abstract} with initial condition $\zeta_{n}(0)=\zeta^0_n$. By applying  It\^o's formula, thanks to \eqref{sgfine1} and \eqref{sm20}, we get
		\begin{equation}\label{sm21}
			\begin{array}{l}
				\ds{\frac{1}{2}d\norm{\zeta_{n}(t)-\zeta_{m}(t)}_{\H}^{2} }\\[10pt]
	\ds{\hslp=\Inner{\A(\zeta_{n}(t))-\A(\zeta_{m}(t)),\zeta_{n}(t)-\zeta_{m}(t)}_{\H}dt+\frac{1}{2}\norm{\Sigma(\zeta_{n}(t))-\Sigma(\zeta_{m}(t))}_{\mathcal{L}_{2}(H_{Q},\H)}^{2}dt }\\
				[10pt]
				\ds{\hsp\hslp +\Inner{\zeta_{n}(t)-\zeta_{m}(t),\big[\Sigma(\zeta_{n}(t))-\Sigma(\zeta_{m}(t))\big]dw^{Q}(t)}_{\H} }\\
				[10pt]
			\ds{\hslp\leq c\norm{\zeta_{n}(t)-\zeta_{m}(t)}_{\H}^{2}dt+\Inner{\zeta_{n}(t)-\zeta_{m}(t),\big[\Sigma(\zeta_{n}(t))-\Sigma(\zeta_{m}(t))\big]dw^{Q}(t)}_{\H} }.
			\end{array}
		\end{equation}
		Due to \eqref{sigma-lip} we have
	\[	\begin{array}{l}
		\ds{\mathbb{E}\sup_{s \in\,[0,t]}\left \vert\int_0^s \Inner{\zeta_{n}(r)-\zeta_{m}(r),\big[\Sigma(\zeta_{n}(r))-\Sigma(\zeta_{m}(r))\big]dw^{Q}(r)}_{\H}\right\vert	}\\[14pt]
		\ds{\leq c\,\mathbb{E}\left(\int_0^t\Vert \zeta_n(s)-\zeta_m(s)\Vert^4_{\mathcal{H}}\,ds\right)^{\frac 12}\leq \frac 14\, \mathbb{E}\sup_{s \in\,[0,t]}\Vert \zeta_n(s)-\zeta_m(s)\Vert_{\mathcal{H}}^2+c\,\int_0^t\mathbb{E}\Vert \zeta_n(s)-\zeta_m(s)\Vert_{\mathcal{H}}^2\,ds.}
		\end{array}\]
Thus, if we first integrate both sides in \eqref{sm21} with respect to time, and then take the supremum and the expectation, we get
	\begin{equation*}
		\E\sup_{s \in\,[0,t]} \norm{\zeta_{n}(s)-\zeta_{m}(s)}_{\H}^{2}\leq \norm{\zeta^0_n-\zeta^0_m}_{\H}^{2}+c\,\int_{0}^{t}\E\norm{\zeta_{n}(s)-\zeta_{m}(s)}_{\H}^{2}ds,
	\end{equation*}
	and the Gronwall's lemma gives 
	\[
		\E\sup_{s \in\,[0,t]} \norm{\zeta_{n}(s)-\zeta_{m}(s)}_{\H}^{2}\leq e^{c\, t }\norm{\zeta^0_n-\zeta^0_m}_{\H}^{2},\ \ \ \ \   t\geq0,
	\]
	for some constant $c$.
	In particular, this implies that the sequence $\{\zeta_{n}\}_{n\in\mathbb{N}}$ is Cauchy in the space $L^{2}(\Omega;C([0,T];\H))$, so there exists a limit $\zeta\in L^{2}(\Omega;C([0,T];\H))$. It is easy to see that the limit $\zeta$ does not depend on the choice of the sequence $\{\zeta^0_n\}\subset \H_{1}$, which implies the uniqueness of generalized solutions. Finally, by using a similar argument as above, we obtain \eqref{Feller_z}.
	\end{proof}

	\begin{Remark}\label{coincide}
		{\em When $\zeta^0\in \H_{1}$, the unique generalized solution $\zeta_{\mu}$ of equation \eqref{abstract} coincides with its unique classical solution.}
	\end{Remark}

Now, we consider equation \eqref{system} in the space $\mathcal{H}$.

\begin{Definition}
	For every $\mu,T>0$ and every $(\mathfrak{u}_0,\mathfrak{v}_0)\in \mathcal{H}$, we say that the process  $z_{\mu}\in L^{2}(\Omega;C([0,T];\mathcal{H}))$ is a generalized solution of system \eqref{system} if there exists a sequence $\{\mathfrak{u}_{0,n},\mathfrak{v}_{0,n}\}_{n\in\mathbb{N}}\subset\H_{1}$ converging to $(\mathfrak{u}_0,\mathfrak{v}_0)$ in $\H$, as $n\to+\infty$, such that
	\begin{equation*}
		\lim_{n\to +\infty}z_{\mu,n}=z_\mu \ \ \text{in}\ \  L^{2}(\Omega;C([0,T];\mathcal{H})),
	\end{equation*}
	where $z_{\mu, n}\in L^{2}(\Omega;C([0,T];\H_{1}))$ is the unique solution of equation \eqref{system} with initial conditions $(\mathfrak{u}_{0,n},\mathfrak{v}_{0,n})$. 
\end{Definition}

As we have seen above, due to Hypothesis \ref{Hypothesis2},  $z_\mu=(u_{\mu},v_{\mu})$ is a generalized solution to \eqref{system} with initial condition $\mathfrak{z}_0=(\mathfrak{u}_0,\mathfrak{v}_0)$ if and only if  $\zeta_{\mu}=\big(u_{\mu},\sqrt{\mu}v_{\mu}+g(u_{\mu})/\sqrt{\mu}\big)$ is a generalized solution for  system \eqref{abstract} with initial condition $\zeta_{\mu}(0)=(\mathfrak{u}_0,\sqrt{\mu}\mathfrak{v}_0+g(\mathfrak{u}_0)/\sqrt{\mu})$. In this case, we have 
\begin{equation}\label{transform}
	u_{\mu}=\Pi_{1}\zeta_{\mu},\ \ \ \ v_{\mu}=\frac{1}{\mu}\Big(-g(u_{\mu})+\sqrt{\mu}\,\Pi_{2}\zeta_{\mu}\Big),\ \ \ \mu>0.
\end{equation}

Thus, as a consequence of Lemma \ref{wellposedness_gen_z} and  Remark \ref{coincide}, we have the following result.
\begin{Proposition}\label{wellposedness_old}
	Fix $(\mathfrak{u}_0,\mathfrak{v}_0)\in \mathcal{H}$ and assume Hypotheses \ref{Hypothesis1}, \ref{Hypothesis2} and \ref{Hypothesis3}. Then, for every $\mu, T>0$ there exists a unique generalized solution $z_\mu\in L^{2}(\Omega;C([0,T];\mathcal{H}))$ for system \eqref{system}.  In particular, when $(\mathfrak{u}_0,\mathfrak{v}_0)\in \H_{1}$, the unique generalized solution coincides with the unique classical solution.
\end{Proposition}
This allows to introduce the transition semigroup  associated with equation \eqref{system} in $\mathcal{H}$, which will be denoted by ${P}_{t}^{\mu, \mathcal{H}}$. Clearly, if $\varphi \in\,B_b(\mathcal{H})$,   for every $\mu>0$ we have
\[
	{P}_{t}^{\mu, \mathcal{H}_1}\varphi(\mathfrak{z})=P_{t}^{\mu, \mathcal{H}}\varphi(\mathfrak{z}),\ \ \ \ \   \mathfrak{z}\in\H_{1},\ \ \ t\geq 0.
\]

\subsection{Existence of  invariant measures for equation \eqref{system}}

\begin{Proposition}\label{sys_invariant}
	Assume Hypotheses \ref{Hypothesis1}, \ref{Hypothesis2} and \ref{Hypothesis3}. Then, for every $\mu>0$, the semigroup $ P_t^{\mu, \mathcal{H}}$ admits an invariant measure $\nu^{\mathcal{H}}_{\mu}$ in $\mathcal{H}$, with $\text{supp}\,(\nu_{\mu}^{\mathcal{H}})\subset \mathcal{H}_{1}$.
	%Furthermore,
	%\begin{equation}\label{uniform_invariant2}
	%	\int_{\mathcal{H}}\Big(\norm{u}_{H^{1}}^{4}+\mu^{2}\norm{v}_{H}^{4}\Big)\nu_{\mu}(du,dv)\leq C,\ \ \  \mu\in(0,1).
	%\end{equation}
\end{Proposition}

\begin{proof} 	First, if $z_\mu^{\mathfrak{z}_1}$ and $z_\mu^{\mathfrak{z}_2}$ are generalized solutions to system \eqref{system}, with initial conditions $\mathfrak{z}_1, \mathfrak{z}_2\in \mathcal{H}$, respectively, then due to \eqref{Feller_z} and \eqref{transform}, it is easy to see that for every $t\geq0$
	\begin{equation*}
		\E\norm{z_\mu^{\mathfrak{z}_1}(t)-z_\mu^{\mathfrak{z}_2}(t))}_{\mathcal{H}}\leq c_\mu(t)\norm{\mathfrak{z}_1-\mathfrak{z}_2}_{\mathcal{H}},
	\end{equation*} 
	for some $c_\mu(t)>0$. This means that the transition semigroup $P^{\mu, \mathcal{H}}_{t}$ is Feller on $\mathcal{H}$.
	
	Now, for every $\mathfrak{z} \in\,\mathcal{H}$ we introduce the  following family of measures on $\mathcal{H}$
	\begin{equation*}
		\Gamma^\mu_{t}(\mathfrak{z},\cdot):=\frac{1}{t}\int_{0}^{t}(P^{\mu, \mathcal{H}}_{t})^\star\delta_\mathfrak{z}\,dt,\ \ \ \ \ \ t>0,
	\end{equation*}
	and for every $R>0$ we define the set
	\begin{equation*}
		B_{R}:=\Big\{\mathfrak{z}\in \mathcal{H}_{1}:\norm{\mathfrak{z}}_{\mathcal{H}_{1}}\leq R\Big\}.
	\end{equation*}
Then, from \eqref{sys_ene1} and \eqref{sys_ene2} with $\mu=1$, we have
	\begin{equation*}
		\Gamma^\mu_{t}(0,B_{R}^{c})=\frac{1}{t}\int_{0}^{t}\mathbb{P}\big(\norm{z_\mu^0(s)}_{\H_{1}}>R\big)ds\leq \frac{c}{R^{2}},\ \ \ t>0,\ \ \ \ \ \ R>0,
	\end{equation*}
	and, due to the compactness of the embedding of $\mathcal{H}_1$ into $\mathcal{H}$, this implies that the family of measures $\{\Gamma^\mu_{t}(0,\cdot)\}_{t\geq0}$ is tight in $\mathcal{H}$. By the Prokhorov theorem, there exists some sequence $t_n\uparrow \infty$ such that $\Gamma^\mu_{t_n}(0,\cdot)$ converges weakly, as $n\to+\infty$, to a probability measure $\nu_\mu^{\mathcal{H}}$ that is invariant for  $P^{\mu, \mathcal{H}}_{t}$. Moreover, since
	\begin{equation*}
		\nu_\mu^{\mathcal{H}}(B_{R}^{c})\leq \frac{c}{R^{2}},\ \ \   R>0,
	\end{equation*}
	it follows that $\text{supp}\,(\nu_\mu^{ \mathcal{H}})\subset\mathcal{H}_{1}$.

	%(3) To prove \eqref{uniform_invariant2}, we recall the function 
	%\begin{equation*}
	%	\varphi_{\mu}(u,v):=\frac{1}{2}\Big(\norm{u}_{H^{1}}^{2}+\mu\norm{v}_{H}^{2}\Big),
	%\end{equation*}
	%and we have 
	%\begin{equation*}
	%	\begin{array}{ll}
		%	\ds{\mathcal{N}_\mu\big(\varphi_{\mu}^{2}\big)(u,v) }
		%	&\ds{=2\varphi_{\mu}\mathcal{N}_\mu\varphi_{\mu}+\frac{1}{\mu^{2}}\norm{QD_{v}\varphi_{\mu}}_{H^{-1}}^{2} }\\
		%	\vs
		%	&\ds{=\frac{1}{2}\Big(\frac{1}{2\mu}\text{Tr}Q^{2}-\Inner{\gamma(u)v,v}_{H}\Big)\Big(\norm{u}_{H^{1}}^{2}+\mu\norm{v}_{H}^{2}\Big)+\norm{Qv}_{H}^{2} }\\
		%	\vs
		%	&\ds{\leq -\frac{C}{\mu}\Big(\mu^{2}\norm{v}_{H}^{4}\Big)+\frac{c}{\mu}\varphi_{\mu}(u,v),\ \ \ \mu\in(0,1). }
		%	\end{array}
	%\end{equation*}
	%Due to the invariant of $\nu_{\mu}$ in $\mathcal{H}$, we have
	%\begin{equation*}
	%	\int_{\mathcal{H}}\mathcal{N}_\mu\big(\varphi_{\mu}^{2}\big)\nu_{\mu}(u,v)=0,
	%\end{equation*}
	%and then thanks to \eqref{uniform_invariant}, it follows that
	%\begin{equation*}
	%	\int_{\mathcal{H}}\mu^{2}\norm{v}_{H}^{4}\nu_{\mu}(u,v)\leq C,\ \ \ \mu\in(0,1).
	%\end{equation*}
	%By applying the same method to the function $\phi_{\mu}(u,v)$ introduced in part (2), we have
	%\begin{equation*}
	%	\int_{\mathcal{H}}\norm{u}_{H^{1}}^{4}\nu_{\mu}(u,v)\leq C,\ \ \ \mu\in(0,1).
	%\end{equation*}
	%Therefore \eqref{uniform_invariant2} holds.
\end{proof}

\begin{Remark}
{\em 	
Since $\text{supp}\,(\nu_{\mu}^{\mathcal{H}})\subset \H_{1}$ and $\B(\H_{1})\subset \B(\H)$, we have that $\nu^{\mathcal{H}}_\mu$  is also a probability measure on $\H_{1}$. In what follows,  it will be convenient to denote the restriction of $\nu^{\mathcal{H}}_\mu$ to $\mathcal{H}_1$ by  $\nu^{\mathcal{H}_1}_{\mu}$.}
\end{Remark}

\subsection{Statement of the main result}

Given a lower semicontinuous metric $\alpha$ on $H^{-1}$, it is possible to introduce the  distance $\mathcal{W}_{\alpha}:\mathcal{P}(H^{-1})\times\mathcal{P}(H^{-1})\to[0,+\infty]$ defined by
\begin{equation*}
	\mathcal{W}_{\alpha}(\nu_{1},\nu_{2})=\sup_{[\varphi]_{\text{Lip}_{H^{-1}}^{\alpha}}\leq 1}\,\left\lvert \int_{H^{-1}} \varphi(\mathfrak{r})\nu_{1}(d\mathfrak{r})-\int_{H^{-1}}\varphi(\mathfrak{r})\nu_{2}(d\mathfrak{r})\right\rvert,\end{equation*}
where 
\begin{equation*}
	[\varphi]_{\text{Lip}_{H^{-1}}^{\alpha}}=\sup_{\substack{\mathfrak{r}_1, \mathfrak{r}_2 \in\,H \\ \mathfrak{r}_1\neq \mathfrak{r}_2}}\frac{\abs{\varphi(\mathfrak{r}_1)-\varphi(\mathfrak{r}_2)}}{\alpha(\mathfrak{r}_1,\mathfrak{r}_2)}.
\end{equation*}
Notice that the following Kantorovich-Rubinstein identity holds
\begin{equation}\label{KR-id}
	\mathcal{W}_{\alpha}(\nu_{1},\nu_{2})=\inf_{\lambda \in\,\mathcal{C}(\nu_1,\nu_2)}\int\int\alpha(\mathfrak{r}_1,\mathfrak{r}_2)\,\lambda(d\mathfrak{r}_1,d\mathfrak{r}_2),
\end{equation}
where $\mathcal{C}(\nu_1, \nu_2)$ is the set of all couplings of $(\nu_{1},\nu_{2})$. Notice that it is possible to prove that the infimum above is attained at some $\bar{\lambda}$.

Throughout the rest of this paper, we will prove that the following result holds.

\begin{Theorem}
\label{main-teo}	 Assume Hypotheses \ref{Hypothesis1} to \ref{Hypothesis_control}, and define  
\[\alpha(\mathfrak{u}_1,\mathfrak{u}_2):=\norm{\mathfrak{u}_1-\mathfrak{u}_2}_{H^{-1}},\ \ \ \ \ \ \mathfrak{u}_1, \mathfrak{u}_2\in H^{-1}.\]
Then we have
	\[
\lim_{\mu\to 0} \mathcal{W}_\alpha\left(\Pi_1 \nu_\mu^{\mathcal{H}},		\nu\right)=0,
	\]
	where $\nu$ is the unique invariant measure for $P^H_t$, the transition semigroup associated to the limiting equation \eqref{limiting_problem_intro}. Moreover, 
	\[\lim_{\mu\to 0}\Pi_1 \nu_\mu^{\mathcal{H}}=\nu,\ \ \ \ \ \text{weakly in}\ H.\]

\end{Theorem}

%%%%%%%%%%%%%%%%%%%%%%%%%%%%%%%%%%%%%%%%%%%%%%%%%%%%%%%%%%

\section{Some uniform bounds}\label{sec_sys_invariant}

In what follows, we are going to show that for every $\mu>0$ equation \eqref{system} has an invariant measure and we will prove some uniform bounds for the moments of such family of invariant measures.

To this purpose, we need to start with suitable uniform bounds for 
 the solution $(u_{\mu},v_{\mu})$ of system \eqref{system}. Some of them   have been already proved in \cite[Proposition 4.2, Remark 4.3]{cerraixi}. In what follows, we show how those bounds depend on time and on  random initial conditions in $L^2(\Omega;\mathcal{H}_1)$.

\begin{Lemma}\label{enegy}
	  Assume Hypotheses \ref{Hypothesis1}, \ref{Hypothesis2} and \ref{Hypothesis3}, and fix  $(\xi,\eta)\in L^{2}(\Omega;\mathcal{H}_{1})$. For every $\mu, T>0$, let $(u_{\mu},v_{\mu})\in L^{2}(\Omega;C([0,T];\mathcal{H}_{1}))$ be the unique solution to system \eqref{system} with initial conditions $(\xi, \eta)$. Then there exist two constants $\mu_{0}\in (0,1)$ and $c>0$, independent of  $T>0$, such that for every $\mu\in(0,\mu_{0})$ and $t \in\,[0,T]$
	\begin{align} \begin{split}\label{sys_ene1}
		\E\sup_{s\in[0,t]}&\norm{u_{\mu}(s)}_{H^{1}}^{2}+\mu\,\E\sup_{s\in[0,t]}\norm{v_{\mu}(s)}_{H}^{2}+\int_{0}^{t}\E\norm{v_{\mu}(s)}_{H}^{2}ds\\[10pt]
		&\hslp\leq c\,\left(\frac{t}{\mu}+1\right)+c\,\left(\E\norm{\xi}_{H^{1}}^{2}+\mu\,\E\norm{\eta}_{H}^{2}\right),
	\end{split} \end{align}
	and 
	\begin{equation}\label{sys_ene2}
		\E\sup_{s\in[0,t]}\norm{u_{\mu}(s)}_{H}^{2}+\int_{0}^{t}\E\norm{u_{\mu}(s)}_{H^{1}}^{2}ds\leq c\,\left(1+t+\E\norm{\xi}_{H}^{2}+\mu\, \E\norm{\xi}_{H^{1}}^{2}+\mu^{2}\,\E\norm{\eta}_{H}^{2}\right).
	\end{equation}
\end{Lemma}

\begin{proof}
	 Let $(\mathfrak{u}_0,\mathfrak{v}_0)\in L^{2}(\Omega;\mathcal{H}_{1})$ and let $(u_{\mu},v_{\mu})\in L^{2}(\Omega;C([0,T];\mathcal{H}_{1}))$ be the unique solution to system \eqref{system}. By proceeding as in the proof of \cite[Lemma 4.2]{cerraixi}, we have for every $\mu\in(0,1)$ 
	\begin{equation}\label{prelim_ene1}
		\begin{array}{l}
			\ds{\E\sup_{s\in[0,t]}\norm{u_{\mu}(s)}_{H^{1}}^{2}+\mu\,\E\sup_{s\in[0,t]}\norm{v_{\mu}(s)}_{H}^{2}+\int_{0}^{t}\E\norm{v_{\mu}(s)}_{H}^{2}ds }\\
			\vs
			\ds{\ \ \ \ \ \ \ \ \ \ \leq c\,\left(\frac{t}{\mu}+\Big(1+\E\norm{\xi}_{H^{1}}^{2}+\mu\E\norm{\eta}_{H}^{2}\Big)+\int_{0}^{t}\E\norm{u_{\mu}(s)}_{H}^{2}ds\right) }.
		\end{array}
	\end{equation}
	Moreover, by proceeding as in the proof of \cite[Lemma 4.1]{cerraixi}, we have $\P$-a.s. 
	\begin{equation*}
		\begin{array}{l}
			\ds{\frac{\gamma_{0}}{4}\norm{u_{\mu}(t)}_{H}^{2} }
			\ds{\leq c\,\left(\norm{\xi}_{H}^{2}+\mu^{2}\norm{\eta}_{H}^{2}\right)+c\,\mu^{2}\norm{v_{\mu}(t)}_{H}^{2}+\mu\int_{0}^{t}\norm{v_{\mu}(s)}_{H}^{2}ds }\\
			\vs
			\ds{\quad \quad-\int_{0}^{t}\norm{u_{\mu}(s)}_{H^{1}}^{2}ds+\int_{0}^{t}\Inner{F(u_{\mu}(s)),u_{\mu}(s)}_{H}ds+\int_{0}^{t}\Inner{u_{\mu}(s),\sigma(u_{\mu}(s))dw^{Q}(s)}_{H} }.
		\end{array}
	\end{equation*}
	Due to \eqref{sm1}, for every $\delta>0$ we have
	\begin{equation*}
		\E\sup_{s\in[0,t]}\left\lvert \int_{0}^{s}\Inner{F(u_{\mu}(r)),u_{\mu}(r)}_{H}dr\right\rvert\leq \frac{1}{\alpha_{1}}(L_f+\delta)\int_{0}^{t}\E\norm{u_{\mu}(s)}_{H^{1}}^{2}ds+c_\delta t.
	\end{equation*}
	Moreover, due to \eqref{sgfine2} we have
	\begin{equation*}
		\begin{array}{l}   \ds{\E\sup_{s\in[0,t]}\left\lvert\int_{0}^{s}\Inner{u_{\mu}(r),\sigma(u_{\mu}(r))dw^{Q}(r)}_{H}\right\rvert}\\[14pt]
	\ds{\quad \quad \quad \leq c\,\left(\mathbb{E}\int_0^t\Vert u_\mu(s)\Vert^2_{H}\,ds\right)^{\frac 12}\leq \frac{\delta}{\alpha_{1}}\int_{0}^{t}\E\norm{u_{\mu}(t)}_{H^{1}}^{2}dt+c_\delta.}		\end{array}
\end{equation*}
	According  to \eqref{L_F_condition},  we can fix $\delta>0$ such that
	\[\frac 1{\alpha_1}\left(L_f+2\,\delta\right)<1,\]
	and this yields
	\begin{equation}\label{prelim_ene2}
		\begin{array}{ll}
		\ds{\E\sup_{s\in[0,t]}\norm{u_{\mu}(s)}_{H}^{2}+\int_{0}^{t}\E\norm{u_{\mu}(s)}_{H^{1}}^{2}ds}\\
		\vs
		\ds{\quad \quad \leq c\,\Big(1+t+\E\norm{\xi}_{H}^{2}+\mu^{2}\E\norm{\eta}_{H}^{2}\Big)+c\,\mu^{2}\,\E\sup_{s\in[0,t]}\norm{v_{\mu}(s)}_{H}^{2}+c\,\mu\int_{0}^{t}\E\norm{v_{\mu}(s)}_{H}^{2}ds}.
		\end{array}
	\end{equation}
	Thus \eqref{sys_ene2} holds by combining \eqref{prelim_ene2} with \eqref{prelim_ene1}. Finally, by combining \eqref{sys_ene2} with \eqref{prelim_ene1}, we complete the proof of \eqref{sys_ene1}.
		 
\end{proof}

\begin{Lemma}
	Let $\{(\xi_{\mu},\eta_{\mu})\}_{\mu \in\,(0,1)}\subset L^{2}(\Omega;\H_{1})$ be a family of random variables such that
	\begin{equation}\label{initial_condition}
		\sup_{\mu\in(0,1)}\E\Big(\norm{\xi_{\mu}}_{H^{1}}^{2}+\mu\norm{\eta_{\mu}}_{H}^{2}\Big)<\infty.
	\end{equation}
	If $(u_{\mu},v_{\mu})\in L^{2}(\Omega;C([0,T];\mathcal{H}_{1}))$ is the solution to system \eqref{system} with initial condition $(\xi_{\mu},\eta_{\mu})$, then there exist $\mu_{T}\in(0,\mu_{0})$ and $c_{T}>0$ such that for every $\mu\in(0,\mu_{T})$ 
	\begin{equation}\label{sys_ene3}
		 \E\sup_{t\in[0,T]}\Big(\norm{u_{\mu}(t)}_{H^{1}}^{2}+\mu\norm{v_{\mu}(t)}_{H}^{2}\Big)\leq \frac{c_{T}}{\sqrt{\mu}}+\Big(\E\norm{\xi_{\mu}}_{H^{1}}^{2}+\mu\,\E\norm{\eta_{\mu}}_{H}^{2}\Big).
	\end{equation} 
\end{Lemma}

\begin{proof}
	If for every $\mu\in(0,\mu_{0})$ and $t\in[0,T]$, we define 
	\begin{equation*}
		L_{\mu}(t):=\norm{u_{\mu}(t)}_{H^{1}}^{2}+\mu\norm{v_{\mu}(t)}_{H}^{2}-\Big(\norm{\xi_{\mu}}_{H^{1}}^{2}+\mu\norm{\eta_{\mu}}_{H}^{2}\Big),
	\end{equation*} 
	then \eqref{sys_ene3} is equivalent to
	\begin{equation}\label{claim}
		\sqrt{\mu}\ \mathbb{E}\sup_{t\in[0,T]}L_{\mu}(t)\leq c_{T},\ \ \ \ \ \ \ \ \mu\in(0,\mu_{T}),
	\end{equation}
	for some constants $\mu_{T}\in(0,1)$ and $c_{T}>0$.
	
Now, if we assume (\ref{claim}) is not true, there exists a sequence $(\mu_{k})_{k\in\mathbb{N}}\subset(0,\mu_{0})$ converging to $0$, as $k\to\infty$, such that
	\begin{equation}\label{contra}
		\lim_{k\to\infty}\sqrt{\mu_{k}}\ \mathbb{E}\sup_{t\in[0,T]}L_{\mu_{k}}(t)=+\infty.
	\end{equation}
	For every $k\in\mathbb{N}$, the mapping $t\mapsto L_{\mu_k}(t)$ is  continuous $\mathbb{P}$-a.s., so that  there exists a random time $t_{k}\in[0,T]$ such that
	\begin{equation*}
		L_{\mu_{k}}(t_{k})=\sup_{t\in[0,T]}L_{\mu_{k}}(t).
	\end{equation*}
	As a consequence of  It\^o's formula,  we have
	\[
    \begin{aligned}
    &\frac 12 d \left(\Vert  u_\mu(t)\Vert _{H^1}^2+\mu\, \Vert  _\mu(t)\Vert _H^2 \right)\\[10pt]
    &=\le(\langle F(u_\mu(t),v_\mu(t)\rangle_H-\langle \gamma(u_\mu(t))v_\mu(t),v_\mu(t)\rangle_H +\frac{1}{2\mu}\Vert \si(u_\mu(t))\Vert^2_{\mathcal{L}_2(H_Q,H)}\r)dt\\[10pt]
    &\hsp\hslp +\langle v_\mu(t),\si(u_\mu(t))dw^Q(t)\rangle_H\\[10pt]
    &\leq \le(c\,\le(\Vert  u_\mu(t)\Vert _{H}^2 +1\r)-\frac {\gamma_0}2 \Vert  v_\mu(t)\Vert _H^2+\frac{\si_\infty^2}{2\mu}\r)dt+\langle v_\mu(t),\si(u_\mu(t))dw^Q(t)\rangle_H.
\end{aligned}
\]
Hence, 	if $s$ is any random time  such that $\mathbb{P}(s\leq t_{k})=1$,  we have 
	\begin{equation*}
		L_{\mu_{k}}(t_{k})-L_{\mu_{k}}(s)\leq \frac{\sigma_{\infty}^{2}}{\mu_k}(t_{k}-s)+c\int_{s}^{t_{k}}\Big(1+\norm{u_{\mu_{k}}(r)}_{H}^{2}\Big)dr+2\big(M_{k}(t_{k})-M_{k}(s)\big),
	\end{equation*}
	where 
	\begin{equation*}
		M_{k}(t):=\int_{0}^{t}\Inner{v_{\mu_{k}}(r),\sigma(u_{\mu_{k}}(r))dw^{Q}(r)}_{H}.
	\end{equation*}

	If we define 
	\begin{equation*}
	 U_{k}:=c\int_0^T\norm{u_{\mu_{k}}(t)}_{H}^{2}\,dt,\ \ \ \ \ \ \ M_{k}:=\sup_{t\in[0,T]}\abs{M_{k}(t)},
	\end{equation*}
	this implies that there exists some constant $\lambda>0$, independent of $k$, such that
	\begin{equation*}
		L_{\mu_{k}}(t_{k})-L_{\mu_{k}}(s)\leq \frac{\lambda}{\mu_k}(t_{k}-s)+U_{k}+4M_{k},
	\end{equation*}
	and  since $L_{\mu_{k}}(s)=0$, if we take $s=0$ we get
	\begin{equation*}
		t_{k}\geq \frac{\mu_{k}}{\lambda}\Big(L_{\mu_{k}}(t_{k})-U_{k}-4M_{k}\Big)=:\frac{\mu_{k}\theta_{k}}{\lambda}.
	\end{equation*}
Now, on the set $E_{k}:=\big\{\theta_{k}>0\big\}$, we fix an arbitrary  $s\in\big[t_{k}-\mu_{k}\theta_{k}/(2\lambda),t_{k}\big]$ and we have 
	\begin{equation}\label{positive}
		L_{\mu_{k}}(s)\geq L_{\mu_{k}}(t_{k})-\frac{1}{2}\theta_{k}-U_{k}-4M_{k}=\frac{1}{2}\theta_{k}>0.
	\end{equation}
	Hence, if we define 
	\begin{equation*}
		I_{k}:=\int_{0}^TL^+_{\mu_{k}}(s)\,ds,
	\end{equation*}
	due to \eqref{positive} we have
	\begin{equation*}
		I_{k}\geq \int_{t_{k}-\frac{\mu_{k}\theta_{k}}{2\lambda}}^{t_{k}}L_{\mu_{k}}(s)ds\geq \frac{\mu_{k}}{4\lambda}\,\theta_k^{2},\ \ \ \ \ \ \ \text{on}\ E_k,
	\end{equation*}	
	so that
	\begin{equation}\label{contra1}
		 \mathbb{E}(I_{k};E_{k})\geq \mathbb{E}\Bigg(\frac{\mu_{k}}{4\lambda}\theta_k^{2};E_{k}\Bigg).
	\end{equation}
	
 Now, according to \eqref{sys_ene1}, \eqref{sys_ene2} and \eqref{initial_condition}
 	\begin{equation*}
 		\E\,U_{k}\leq c\Big(1+T+\E\,\norm{\xi_{\mu_{k}}}_{H}^{2}+\mu_{k}\,\E\,\norm{\xi_{\mu_{k}}}_{H^{1}}^{2}+\mu_{k}^{2}\,\E\,\norm{\eta_{\mu_{k}}}_{H}^{2}\Big)\leq c_T,
 	\end{equation*}
 	and 
	\begin{equation*}
		\E\,M_{k}\leq c \Big(\int_{0}^{T}\E\norm{v_{\mu_{k}}(t)}_{H}^{2}dt\Big)^{\frac{1}{2}}\leq c\left(1+\frac{T}{\mu_{k}}+\E\norm{\xi_{\mu_{k}}}_{H^{1}}^{2}+\mu_{k}\,\E\norm{\eta_{\mu_{k}}}_{H}^{2}\right)^{\frac{1}{2}}\leq c_T\Big(1+\frac{1}{\mu_{k}}\Big)^{\frac{1}{2}},
	\end{equation*}
	so that
	\begin{equation*}
		\limsup_{k\to\infty}\sqrt{\mu_{k}}\left(\mathbb{E}\,U_k+4\,  \mathbb{E}\,M_{k}\right)<+\infty.
	\end{equation*}
	Thanks to (\ref{contra}) this gives	
	\begin{equation*}
			\lim_{k\to\infty}\sqrt{\mu_{k}}\ \mathbb{E}(\theta_{k})=+\infty,
	\end{equation*}
	and hence
	\begin{equation}
		\label{contra2}
		\lim_{k\to\infty}\sqrt{\mu_{k}}\ \mathbb{E}(\theta_{k};E_{k})=+\infty.
	\end{equation}
	
	Now, according to \eqref{contra1}, we have
	\begin{equation*}
		\mathbb{E}(I_{k};E_{k})\geq \frac{\mu_k}{4\lambda}\mathbb{E}\big(\theta_{k}^{2};E_{k}\big)\geq \frac{\mu_k}{4\lambda}\left(\E( \theta_{k};E_{k})\right)^{2},
	\end{equation*}
and due to \eqref{contra2}, this implies
	\[\lim\limits_{k\to\infty}\mathbb{E}(I_{k};E_{k})=+\infty.\]	
	However, as a consequence of \eqref{sys_ene1}, \eqref{sys_ene2} and \eqref{initial_condition}, we have
	\begin{equation*}
		\sup_{k\in\mathbb{N}}\,\E\,I_{k}\leq\sup_{k\in\mathbb{N}}\,\int_{0}^{T}\E\,\abs{L_{\mu_{k}}(s)}ds\leq c_T\,\sup_{k\in\mathbb{N}}\Big(1+T+\E\norm{\xi_{\mu_{k}}}_{H^{1}}^{2}+\mu_{k}\E\norm{\eta_{\mu_{k}}}_{H}^{2}\Big)<+\infty,
	\end{equation*}
	 and this gives a contradiction, since $\E(I_{k};E_{k})\leq \E(I_{k})$ for every $k\in\mathbb{N}$.
	  In particular, this means that claim \eqref{claim} is true, and \eqref{sys_ene3} holds.
	  
\end{proof}

\begin{Lemma}\label{sys_invariant_H1}
	For every $\mu>0$, if $\nu_{\mu}^{\mathcal{H}}\in \mathcal{P}(\H)$ is any invariant measure for $ P_t^{\mu, \mathcal{H}}$ supported in $\H_{1}$, then ${\nu}_{\mu}^{\mathcal{H}_1}\in \mathcal{P}(\H_{1})$ is invariant for ${P}_{t}^{\mu, \mathcal{H}_1}$. Moreover,
		\begin{equation}\label{uniform_invariant1}			\sup_{\mu \in\,(0,1)}\,\int_{\mathcal{H}_{1}}\Big(\norm{\mathfrak{u}}_{H^{1}}^{2}+\mu\norm{\mathfrak{v}}_{H}^{2}\Big){\nu}^{\mathcal{H}_1}_{\mu}(d\mathfrak{u},d\mathfrak{v})<\infty.
	\end{equation}
\end{Lemma}

\begin{proof} First, we show the invariance of $\nu_{\mu}^{\mathcal{H}_1}$ for ${P}_{t}^{\mu, \mathcal{H}_1}$. Due to the invariance of  $\nu_{\mu}^{ \mathcal{H}}$ in $\mathcal{H}$, for every $\varphi\in C_{b}(\H)$ we have 
	\begin{equation*}
		\int_{\H} P_t^{\mu, \mathcal{H}}\varphi(\mathfrak{z})\nu_{\mu}^{ \mathcal{H}}(d\mathfrak{z})=\int_{\H}\varphi(\mathfrak{z})\nu_{\mu}^{ \mathcal{H}}(d\mathfrak{z}).
	\end{equation*}
	Thus, since $\text{supp}\,(\nu_{\mu}^{ \mathcal{H}})\subset \H_{1}$ and $\B(\H_{1})\subset \B(\H)$, for every $\varphi\in C_{b}(\H)$ we get
	\begin{equation*}
		\int_{\H_{1}} P_t^{\mu, \mathcal{H}}\varphi(\mathfrak{z})\nu_{\mu}^{\mathcal{H}_1}(d\mathfrak{z})=\int_{\H_{1}}\varphi(\mathfrak{z})\nu_{\mu}^{\mathcal{H}_1}(d\mathfrak{z}).
	\end{equation*}
	If  $(\hat{e}_i)_{i \in\,\mathbb{N}}\subset \mathcal{H}_1$ is an orthonormal basis of $\mathcal{H}$, for every $n\in\mathbb{N}$  we denote by $\Pi_n$ the projection of $\H$ onto $\mathcal{H}(n):=\text{span}(\hat{e}_1,\ldots,\hat{e}_n)$. We have that $\Pi_n:\H\to \H_{1}$ is continuous and 
	\[\Vert \pi_n h\Vert_{\mathcal{H}_1}\leq c_n\,\Vert h\Vert_{\mathcal{H}},\ \ \ \ h \in\,\mathcal{H},\ \ \ \ \ \ 		\lim_{n\to\infty}\norm{\Pi_nh-h}_{\H_{1}}=0,\ \ \ \ \ \   h\in \H_{1}.
	\]
	Hence, if for any $\varphi\in C_{b}(\H_{1})$ and $n\in\mathbb{N}$,  we define $\varphi_{n}:=\varphi\circ \Pi_n$,   we have $\varphi_{n}\in C_{b}(\H)$ and 
	\begin{equation*}
		\lim_{n\to\infty}\abs{\varphi_{n}(h)-\varphi(h)}=0,\ \ \   h\in \H_{1}.
	\end{equation*}
	
	For every $n\in\mathbb{N}$, we have $\sup_{n\in\mathbb{N}}\norm{\varphi_{n}}_\infty\leq \norm{\varphi}_\infty$, and  the dominated convergence theorem implies that 
	for any given $\mu>0$ and $\varphi\in C_{b}(\H_{1})$	\begin{equation*}
		\lim_{n\to\infty} P_t^{\mu, \mathcal{H}}\varphi_{n}(\mathfrak{z})=\lim_{n\to\infty}\E\,\varphi_{n}\big(z_{\mu}^{\mathfrak{z}}(t)\big)=\E\,\varphi\big(z_{\mu}^{\mathfrak{z}}(t)\big)=P^{\mu, \mathcal{H}_1}_{t}\varphi(\mathfrak{z}),\ \ \   \mathfrak{z}\in \H_{1},\ \ \ \ \ \ t\geq 0.
	\end{equation*}
In particular, by taking the limit as $n$ goes to infinity in both sides of 
	\begin{equation*}
		\int_{\H_{1}} P_t^{\mu, \mathcal{H}}\varphi_{n}(\mathfrak{z})\nu_{\mu}^{\mathcal{H}_1}(d\mathfrak{z})=\int_{\H_{1}}\varphi_{n}(\mathfrak{z})\nu_{\mu}^{\mathcal{H}_1}(d\mathfrak{z}),\ \ \   \varphi\in C_{b}(\H_{1}),
	\end{equation*}
we conclude that
	\begin{equation*}
		\int_{\H_{1}}{P}^{\mu, \mathcal{H}_1}_{t}\varphi(\mathfrak{z})\nu_{\mu}^{\mathcal{H}_1}(d\mathfrak{z})=\int_{\H_{1}}\varphi(\mathfrak{z})\nu_{\mu}^{\mathcal{H}_1}(d\mathfrak{z}),\ \ \ \ \ \    \varphi\in C_{b}(\H_{1}),
	\end{equation*}
and this implies the invariance of  $\nu_{\mu}^{\mathcal{H}_1}$.
	
	Next, 
	in order to prove \eqref{uniform_invariant1}, we consider the Komolgov operator associated to ${P}_{t}^{\mu, \mathcal{H}_1}$ in $\H_{1}$
	\begin{equation*}
		\begin{array}{ll}
			\ds{\mathcal{N}_\mu\varphi(\mathfrak{u},\mathfrak{v})}
			&\ds{=\frac{1}{2\mu^{2}}\text{Tr}_{H}\Big[\big(\sigma(\mathfrak{u})Q\big)(\sigma(\mathfrak{u})Q)^{\ast}D^{2}_{\mathfrak{v}}\varphi(\mathfrak{u},\mathfrak{v})\Big]+\Inner{\mathfrak{v},D_{\mathfrak{u}}\varphi(\mathfrak{u},\mathfrak{v}) }_{H^{1}} }\\
			\vs
			&\ds{\quad +\frac{1}{\mu}\Inner{A\mathfrak{u}-\gamma(\mathfrak{u})v+F(\mathfrak{u}), D_{\mathfrak{v}}\varphi(\mathfrak{u},\mathfrak{v}) }_{H}. }
		\end{array}
	\end{equation*}	
If, with the notations of Section \ref{assumption}, we define
\begin{equation*}
		\varphi_{\mu}(\mathfrak{u},\mathfrak{v}):=\frac{1}{2}\Big(\norm{\mathfrak{u}}_{H^{1}}^{2}+\mu\norm{\mathfrak{v}}_{H}^{2}\Big)-\int_{\mathcal{O}}\mathfrak{f}(x,\mathfrak{u}(x))dx=\frac{1}{2}\Big(\norm{\mathfrak{u}}_{H^{1}}^{2}+\mu\norm{\mathfrak{v}}_{H}^{2}\Big)-\Lambda(\mathfrak{u}),
	\end{equation*}
due to \eqref{sm2} we have
\[D_\mathfrak{u}\varphi_\mu(\mathfrak{u},\mathfrak{v})=(-A)\mathfrak{u}-f(\cdot,\mathfrak{u}),\ \ \ \ \ D_\mathfrak{v}\varphi_\mu(\mathfrak{u},\mathfrak{v})=\mu\,\mathfrak{v},\ \ \ \ \ D^2_\mathfrak{v}\varphi_\mu(\mathfrak{u},\mathfrak{v})=\mu\,I_H.\]
Then, we have	\begin{equation}\label{sm3}
		\begin{array}{l}
			\ds{\mathcal{N}_\mu\varphi_{\mu}(\mathfrak{u},\mathfrak{v})}\\[14pt]
			\ds{\quad =\frac{1}{2\mu}\text{Tr}_{H}\Big[\big(\sigma(\mathfrak{u})Q\big)(\sigma(\mathfrak{u})Q)^{\ast}\Big]+\Inner{\mathfrak{v},\mathfrak{u}-(-A)^{-1}F(\mathfrak{u})}_{H^{1}} +\frac{1}{\mu}\Inner{A\mathfrak{u}-\gamma(\mathfrak{u})\mathfrak{v}+F(\mathfrak{u}),\mu \mathfrak{v} }_{H} }\\[14pt]
			\ds{\quad=\frac{1}{2\mu}\norm{\sigma(\mathfrak{u})}_{\mathcal{L}_{2}(H_{Q},H)}^{2}+\Inner{\mathfrak{v},-A\mathfrak{u}-F(\mathfrak{u})}_{H} +\frac{1}{\mu}\Inner{A\mathfrak{u}-\gamma(\mathfrak{u})\mathfrak{v}+F(\mathfrak{u}),\mu \mathfrak{v} }_{H}}\\[14pt]
			\ds{\quad=\frac{1}{2\mu}\norm{\sigma(\mathfrak{u})}_{\mathcal{L}_{2}(H_{Q},H)}^{2}-\Inner{\gamma(\mathfrak{u})\mathfrak{v},\mathfrak{v}}_{H}
			\leq \frac{1}{2\mu}\norm{\sigma(\mathfrak{u})}_{\mathcal{L}_{2}(H_{Q},H)}^{2}-\gamma_{0}\norm{\mathfrak{v}}_{H}^{2} }.
		\end{array}
	\end{equation}
	By the invariance of $\nu_{\mu}^{\mathcal{H}_1}$ in $\mathcal{H}_{1}$, we have
	\begin{equation*}
		\int_{\mathcal{H}_1}\mathcal{N}_\mu\varphi_{\mu}(\mathfrak{u},\mathfrak{v})\,\nu_{\mu}^{\mathcal{H}_1}(d\mathfrak{u},d\mathfrak{v})=0,
	\end{equation*}
	and thus, due to \eqref{sm3} and \eqref{sgfine2}, 
	\begin{equation}\label{uniform1}
		\sup_{\mu \in\,(0,1)}\mu\int_{\mathcal{H}_1}\norm{\mathfrak{v}}_{H}^{2}\nu_{\mu}^{\mathcal{H}_1}(d\mathfrak{u},d\mathfrak{v})<\infty.
	\end{equation}
	Next, we consider the function 
	\begin{equation*}
		\psi_{\mu}(\mathfrak{u},\mathfrak{v}):=\frac{1}{2}\Big(\mu\norm{\mathfrak{u}}_{H^{1}}^{2}+\norm{g(\mathfrak{u})+\mu \mathfrak{v}}_{H}^{2}\Big).
	\end{equation*}
We have
\[D_\mathfrak{u}\psi(\mathfrak{u},\mathfrak{v})=\mu\,(-A)\mathfrak{u}+\gamma(\mathfrak{u})\left(g(\mathfrak{u})+\mu\,\mathfrak{v}\right),\ \ \ D_\mathfrak{v} \psi(\mathfrak{u},\mathfrak{v})=\mu\left(g(\mathfrak{u})+\mu\,\mathfrak{v}\right),\ \ \  D^2_\mathfrak{v}\psi(\mathfrak{u},\mathfrak{v})=\mu^2\,I_H,\]
	so that
	\begin{equation*}
		\begin{array}{ll}
			\ds{\mathcal{N}_\mu\psi_{\mu}(\mathfrak{u},\mathfrak{v})}
			\ds{=\frac{1}{2}\text{Tr}_{H}\Big[\big(\sigma(\mathfrak{u})Q\big)(\sigma(\mathfrak{u})Q)^{\ast}\Big]+\Inner{\mathfrak{v},\mu \mathfrak{u}+(-A)^{-1}\gamma(\mathfrak{u})g(\mathfrak{u})+\mu (-A)^{-1}\gamma(\mathfrak{u})\mathfrak{v}}_{H^{1}} }\\[14pt]
			\ds{\quad \quad \quad \quad \quad \quad \quad \quad \quad \quad \quad+\frac{1}{\mu}\Inner{A\mathfrak{u}-\gamma(\mathfrak{u})\mathfrak{v}+F(\mathfrak{u}),\mu^{2}\mathfrak{v}+\mu g(\mathfrak{u}) }_{H} }\\[14pt]
			\ds{=\frac{1}{2}\norm{\sigma(\mathfrak{u})}_{\mathcal{L}_{2}(H_{Q},H)}^{2}+\mu\Inner{\mathfrak{v},-A\mathfrak{u}+\frac{\gamma(\mathfrak{u})}\mu g(\mathfrak{u})+\gamma(\mathfrak{u})\mathfrak{v}}_{H} +\Inner{A\mathfrak{u}-\gamma(\mathfrak{u})\mathfrak{v}+F(\mathfrak{u}),\mu \mathfrak{v}+g(\mathfrak{u}) }_{H} }\\[14pt]
			\ds{=\frac{1}{2}\norm{\sigma(\mathfrak{u})}_{\mathcal{L}_{2}(H_{Q},H)}^{2}-\Inner{\gamma(\mathfrak{u})\nabla \mathfrak{u},\nabla \mathfrak{u}}_{H}+\mu\Inner{F(\mathfrak{u}),\mathfrak{v}}_{H}+\Inner{F(\mathfrak{u}),g(\mathfrak{u})}_{H} }.
		\end{array}
	\end{equation*}
	Note that for every $\delta>0$ and $\mu \in\,(0,1)$
	\begin{equation*}
		\mu\big\lvert\Inner{F(\mathfrak{u}),\mathfrak{v}}_{H}\big\rvert\leq \mu\norm{F(\mathfrak{u})}_{H}\norm{\mathfrak{v}}_{H}\leq \delta\big(1+\norm{\mathfrak{u}}_{H^{1}}^{2}\big)+c_\delta\,\mu^{2}\norm{\mathfrak{v}}_{H}^{2},	\end{equation*}
	and 
	\begin{equation*}
		\big\lvert\Inner{F(\mathfrak{u}),g(\mathfrak{u})}_{H}\big\rvert\leq \big(L_f\gamma_{1}+\delta\big)\norm{\mathfrak{u}}_{H}^{2}+c_\delta\leq \frac{1}{\alpha_{1}}\big(L_f\gamma_{1}+\delta\big)\norm{\mathfrak{u}}_{H^{1}}^{2}+c_\delta,
	\end{equation*}
	so that, due to the invariance of $\nu_\mu^{\mathcal{H}}$, we have
	\[\begin{aligned}
\gamma_0&\int_{\mathcal{H}_1} \Vert \mathfrak{u}\Vert_{H^1}^2\,\nu_\mu^{\mathcal{H}}\,(d\mathfrak{u},d\mathfrak{v})\\[10pt]
	&\leq \frac{\sigma_\infty^2}2+\left(\frac{L_f\,\gamma_1}{\alpha_1}+2\,\delta\right)\int_{\mathcal{H}_1} \Vert \mathfrak{u}\Vert_{H^1}^2\,\nu_\mu^{\mathcal{H}}\,(d\mathfrak{u},d\mathfrak{v})+c_\delta +c_\delta\,\mu^2 \int_{\mathcal{H}_1} \Vert \mathfrak{v}\Vert_{H}^2\,\nu_\mu^{\mathcal{H}}\,(d\mathfrak{u},d\mathfrak{v}).
		\end{aligned}\]
Thanks to \eqref{L_F_condition}, this implies that we can take $\delta>0$ sufficiently small so that  
\[\frac{L_f\,\gamma_1}{\alpha_1}+2\,\delta<\gamma_0,\]
and then 
	\begin{equation*}\label{uniform2}
\int_{\mathcal{H}_{1}}\norm{\mathfrak{u}}_{H^{1}}^{2}\nu_{\mu}^{\mathcal{H}_1}(d\mathfrak{u},d\mathfrak{v})\leq c\left(1+\mu^{2}\int_{\mathcal{H}_{1}}\norm{\mathfrak{v}}_{H}^{2}\nu_{\mu}^{\mathcal{H}_1}(d\mathfrak{u},d\mathfrak{v})\right),\ \ \  \mu\in(0,1).
	\end{equation*}
	By combining this with \eqref{uniform1}, we complete the proof of \eqref{uniform_invariant1}.
\end{proof}

\begin{Remark}
{\em \begin{enumerate}	
\item[1.] In Proposition \ref{sys_invariant}, we have seen that for every $\mu>0$ the semigroup $ P_t^{\mu, \mathcal{H}}$ admits an invariant measure. Thanks to Lemma \ref{sys_invariant_H1}, this implies that for every $\mu>0$ the transition semigroup ${P}_{t}^{\mu, \mathcal{H}_1}$ admits an invariant measure in $\H_{1}$.
\item[2.] As a consequence of \eqref{uniform_invariant1}, we have 
	\begin{equation}
\label{uniform_invariant}
	\sup_{\mu \in\,(0,1)}	\int_{\mathcal{H}}\left(\norm{\mathfrak{u}}_{H^{1}}^{2}+\mu\norm{\mathfrak{v}}_{H}^{2}\right)\nu_{\mu}^{ \mathcal{H}}(d\mathfrak{u},d\mathfrak{v})<\infty.
	\end{equation}
	\end{enumerate}}
\end{Remark}

%%%%%%%%%%%%%%%%%%%%%%%%%%%%%%%%%%%%%%%%%%%%%%%%%%%%%

\section{The limiting equation}

\label{limiting problem}

In order to study the limiting problem
\begin{equation}\label{limiting_problem}
	\left\{\begin{array}{l}
\displaystyle{\gamma(u(t,x))\partial_t u(t,x)= \Delta u(t,x)+f(x,u(t,x)) -\frac{\gamma'(u(t,x))}{2\gamma^2(u(t,x))} \sum_{i=1}^\infty |\sigma(u(t,\cdot))\,Qe_i(x)|^2 }\\[14pt]
\ds{\quad \quad \quad \quad \quad \quad \quad \quad \quad \quad  \quad+\sigma(u(t,\cdot))\partial_t w^Q(t,x)}\\[14pt]
		\displaystyle{u(0,x)=\mathfrak{u}_0(x), \ \ \ \ \ \ \ u(t,\cdot)\big|_{\partial\mathcal{O}}=0,}
	\end{array}\right.
\end{equation}
 we  consider first the following quasilinear stochastic parabolic equation
\begin{equation}\label{limit_para}
	\le\{\begin{array}{l}
		\ds{\partial_{t}\rho(t,x)=\text{div}\big(b(\rho(t,x))\nabla\rho(t,x)\big)+f_{g}(x,\rho(t,x))+\sigma_{g}(\rho(t,\cdot))\partial_{t}w^{Q}(t,x), }\\[10pt]
		\ds{\rho(0,x)=\mathfrak{r}_0(x),\ \ \ \ \ \ \  \rho(t,\cdot)|_{\partial\mathcal{O}}=0 },
	\end{array}\r.
\end{equation}
where for every $r\in\mathbb{R}$ and $x\in\mathcal{O}$
\begin{equation*}
	b(r):=\frac{1}{\gamma(g^{-1}(r))},\ \ \ \ f_{g}(x,r):=f(x,g^{-1}(r)),
\end{equation*}
and for every $h\in H$
\begin{equation*}
	\sigma_{g}(h):=\sigma\big(g^{-1}\circ h\big).
\end{equation*}
As we explained in \cite{cerraixi}, due to a generalized It\^o's formula, the solutions $u$ and $\rho$ of equations \eqref{limiting problem} and \eqref{limit_para}, respectively, are related by 
\begin{equation}  \label{tras}\rho^{\mathfrak{r}_0}(t)=u^{\mathfrak{u}_0}(t),\ \ \ \ \ t\geq 0,\ \ \ \ \ \mathfrak{r}_0:=g(\mathfrak{u}_0).\end{equation}
From Hypothesis \ref{Hypothesis2}, we know 
\begin{equation*}
	\frac{1}{\gamma_{1}}\leq b(r)\leq \frac{1}{\gamma_{0}},\ \ \   r\in\mathbb{R}.
\end{equation*}
Moreover, if we define
\[F_g(h)(x):=f_g(x,h(x)),\ \ \ \ \ x \in\,\mathcal{O},\]
due to  Hypotheses \ref{Hypothesis1}, \ref{Hypothesis2} and \ref{Hypothesis3}, and due to \eqref{sm3},  for every $h_{1},h_{2}\in H$ we have
\begin{equation*}
	\norm{F_{g}(h_{1})-F_{g}(h_{2})}_{H^{-1}}\leq \frac{1}{\sqrt{\alpha_{1}}}\norm{F_{g}(h_{1})-F_{g}(h_{2})}_{H}\leq \frac{L_f}{\sqrt{\alpha_{1}}\gamma_{0}}\norm{h_{1}-h_{2}}_{H},
\end{equation*}
and
\begin{equation*}
	\norm{\sigma_{g}(h_{1})-\sigma_{g}(h_{2})}_{\mathcal{L}_{2}(H_{Q},H^{-1})}\leq \frac{1}{\sqrt{\alpha_{1}}}\norm{\sigma_{g}(h_{1})-\sigma_{g}(h_{2})}_{\mathcal{L}_{2}(H_{Q},H)}\leq \frac{\sqrt{L_\sigma}}{\sqrt{\alpha_{1}}\gamma_{0}}\norm{h_{1}-h_{2}}_{H}.
\end{equation*}
Moreover,  for every $\delta>0$ 
\begin{equation}\label{Fg2}
	\left\lvert\Inner{F_{g}(h),h}_{H}\right\rvert\leq \left(\frac{L_f}{\gamma_{0}}+\delta\right)\norm{h}_{H}^{2}+c_\delta\leq \frac{1}{\alpha_{1}}\left(\frac{L_f}{\gamma_{0}}+\delta\right)\norm{h}_{H^{1}}^{2}+c_\delta,\ \ \   h\in H^{1},
\end{equation}
and 
\begin{equation}\label{Fg3}
		\left\lvert\Inner{F_{g}(h),h}_{H^{-1}}\right\rvert\leq \frac{1}{\alpha_{1}}\norm{F_{g}(h)}_{H}\norm{h}_{H} \leq \frac{1}{\alpha_{1}}\left(\frac{L_f}{\gamma_{0}}+\delta\right)\norm{h}_{H}^{2}+c_\delta,\ \ \   h\in H.
\end{equation}
Finally, thanks to \eqref{sgfine2} we have
\begin{equation}\label{5.5}
	\norm{\sigma_{g}(h)}_{\mathcal{L}_{2}(H_{Q},H)}\leq \sigma_{\infty},\ \ \ \ \ \ \   h\in H.
\end{equation}

In Appendix \ref{A} we have studied the well-posedness of equation \eqref{limit_para} in the space $H$. We have shown that under Hypotheses \ref{Hypothesis1}, \ref{Hypothesis2} and \ref{Hypothesis3}, for every $T>0$ and every $\mathfrak{r}_0\in L^{2}(\Omega;H)$, there exists a unique  adapted  $\rho\in L^{2}(\Omega;C([0,T];H)\cap L^{2}(0,T;H^{1}))$ such that 
for every test function $\varphi\in C^{\infty}_{0}(\mathcal{O})$
	\[	\begin{array}{ll}
			\ds{\Inner{\rho(t),\varphi}_{H} }
			&\ds{=\Inner{\mathfrak{r}_0,\varphi}_{H}-\int_{0}^{t}\Inner{b(\rho(s))\nabla\rho(s),\nabla\varphi}_{H}ds }\\
			\vs
			&\ds{\quad+\int_{0}^{t}\Inner{F_{g}(\rho(s)),\varphi}_{H}ds+\int_{0}^{t}\Inner{\varphi,\sigma_{g}(\rho(s))dw^{Q}(s)}_{H},\ \ \ \P\text{-a.s.} }
		\end{array}
	\]
	In fact, in Appendix \ref{A} we do not need condition \eqref{L_F_condition} and we just replace Hypothesis \ref{Hypothesis3} with	\[	\sup_{x \in\,\mathcal{O}}\abs{f(x,r)-f(x,s)}\leq c\,\abs{r-s},\ \ \ r,s\in\mathbb{R},\ \ \ \ \ \ \ \sup_{x\in\mathcal{O}}\abs{f(x,0)}<\infty.
	\]

	\begin{Lemma}
Under Hypotheses \ref{Hypothesis1}, \ref{Hypothesis2} and \ref{Hypothesis3},  there exist some $\lambda>0$ and $c>0$, such that for every $t\geq 0$
	\begin{equation}
		\label{lim_ene3}
		\E\norm{\rho(t)}_{H}^{2}\leq c\left(1+e^{-\lambda t}\E\norm{\mathfrak{r}_0}_{H}^{2}\right),\ \ \ \ \ \ \ \E\int_{0}^{t}\norm{\rho(s)}_{H^{1}}^{2}ds\leq c\,\left(t+\E\norm{\mathfrak{r}_0}_{H}^{2}\right).
\end{equation}
	\end{Lemma}

\begin{proof}
	We apply  It\^o's formula and we get
	\begin{equation*}
	\begin{aligned}
\frac{1}{2}&d\norm{\rho(t)}_{H}^{2}\leq -\gamma_{1}^{-1}\norm{\rho(t)}_{H^{1}}^{2}dt+\Inner{F_{g}(\rho(t)),\rho(t)}_{H}dt+
\frac{1}{2}\norm{\sigma_{g}(\rho(t))}_{\mathcal{L}_{2}(H_{Q},H)}^{2}dt\\[10pt]
&\hsp\hsp+\inner{\rho(t),\sigma_{g}(\rho(t))dw^{Q}(t)}_{H}.		\end{aligned}
\end{equation*}
	Then thanks to \eqref{Fg2} and \eqref{L_F_condition}, together with \ref{5.5}, we can find some constant $\lambda>0$ such that
	 \begin{equation*}
	 	\frac{d}{dt}\E\norm{\rho(t)}_{H}^{2}+\lambda\E\norm{\rho(t)}_{H^{1}}^{2}\leq c,
	 \end{equation*}
	  and this allows we complete the proof. 
\end{proof}

Due to the results proved in Appendix \ref{A} and \eqref{tras}, equation \eqref{limiting_problem} is well-posed in the space $L^{2}(\Omega;C([0,T];H)\cap L^{2}(0,T;H^{1}))$. Moreover, due to  estimates \eqref{lim_ene3} and the Lipschitz continuity of $g$ and $g^{-1}$ on $\mathbb{R}$, estimates analogous to \eqref{lim_ene3} holds for the solution $u$.
\begin{Proposition}
	Assume  Hypotheses \ref{Hypothesis1}, \ref{Hypothesis2} and \ref{Hypothesis3}. For every $T>0$ and every $\mathfrak{u}_0\in L^{2}(\Omega;H)$, there exists a unique $u\in L^{2}(\Omega;C([0,T];H)\cap L^{2}(0,T;H^{1}))$ which solves  equation \eqref{limiting_problem} in the following sense
	\begin{align*}
\langle u(t),\psi\rangle_H &=\langle \mathfrak{u}_0,\psi\rangle_H -\int_0^t \le\langle \frac{\nabla u(s)}{\gamma(u(s))}, \nabla\psi \r\rangle_H ds-\int_0^t \le\langle \nabla \le(\frac{1}{\gamma(u(s))} \r)\cdot \nabla u(s),\psi \r\rangle_H  ds\\[10pt]
&\quad +\int_0^t \le\langle \frac{f(u(s))}{\gamma(u(s))},\psi \r\rangle_H ds -\int_0^t \le\langle \frac{\gamma '(u(s))}{2\gamma(u(s))^3} \,\sum_{i=1}^\infty (\si(u(s))Qe_i)^2,\psi \r\rangle_H ds\\[10pt] 
&\quad \quad \quad \quad \quad \quad \quad+\int_0^t \le\langle \frac{\si(u(s))}{\gamma(u(s))} \,d w^Q (s),\psi \r\rangle_H,
\end{align*}
for any $\varphi\in C^\infty_0(\mathcal{O})$. Moreover, for every $t\geq 0$
	\begin{equation*}
		\E\norm{u(t)}_{H}^{2}\leq c\,\left(1+e^{-\lambda t}\E\norm{\mathfrak{u}_0}_{H}^{2}\right),\ \ \ \ \ \ 
		\E\int_{0}^{t}\norm{u(s)}_{H^{1}}^{2}ds\leq c\left(t+\E\norm{\mathfrak{u}_0}_{H}^{2}\right).
	\end{equation*}
	\end{Proposition}

%%%%%%%%%%%%%%%%%%%%%%%%%%%%%%%%%%%%%%%%%%%%%%%%%%%%%%%%%%%%%

\subsection{Well-posedness of equation \eqref{limit_para} in $H^{-1}$} 
Here, we will use the results we have just mentioned about the well-posedness of equation \eqref{limit_para} in $H$, to study its well-posedness in $H^{-1}$.
\begin{Definition}
	For every fixed $\mathfrak{r}_0\in H^{-1}$ and $T>0$, an adapted process $\rho\in L^{2}(\Omega;L^{2}(0,T;H))$ is  a  solution of equation \eqref{limit_para} with initial condition $\rho_0$ if for every  $\varphi\in C^{\infty}_{0}(\mathcal{O})$
	\[
			\Inner{\rho(t),\varphi}_{H} 
			=\Inner{\mathfrak{r}_0,\varphi}_{H}+\int_{0}^{t}\Inner{B(\rho(s)),\Delta\varphi}_{H}ds+\int_{0}^{t}\Inner{F_{g}(\rho(s)),\varphi}_{H}ds+\int_{0}^{t}\Inner{\varphi,\sigma_{g}(\rho(s))dw^{Q}(s)}_{H},	\]
	$\mathbb{P}$-a.s., where 
	\begin{equation*}
		B(r):=\int_{0}^{r}b(s)ds,\ \ \   r\in\mathbb{R}.
	\end{equation*}
\end{Definition}

\begin{Proposition}\label{weak_solution}
	Assume Hypotheses \ref{Hypothesis1}, \ref{Hypothesis2} \ref{Hypothesis3} and \ref{Hypothesis_control}. Then, for every $\mathfrak{r}_{0}\in H^{-1}$ and every $T>0$, there exists a unique  solution
	\begin{equation*}
		\rho^{\mathfrak{r}_{0}}\in L^{2}(\Omega;C([0,T];H^{-1})\cap L^{2}(0,T;H)) 
	\end{equation*}
	to equation \eqref{limit_para}. Moreover, there exist $c, \lambda >0$ independent of $T>0$ such that for every $t \in\,[0,T]$
	\begin{equation}\label{lim_ene1}
		\mathbb{E}\norm{\rho^{\mathfrak{r}_{0}}(t)}_{H^{-1}}^{2}\leq c\left(1+e^{-\lambda t}\,\E\norm{\mathfrak{r}_{0}}_{H^{-1}}^{2}\right),\ \ \ \ \ \ \ \ \E\int_{0}^{t}\norm{\rho^{\mathfrak{r}_{0}}(s)}_{H}^{2}ds\leq c\left(t+\mathbb{E}\norm{\mathfrak{r}_{0}}_{H^{-1}}^{2}\right).
	\end{equation}
\end{Proposition}

\begin{proof} We fix an arbitrary  sequence $\{\mathfrak{r}_{\epsilon}\}_{\epsilon>0}\subset H$ converging to $\mathfrak{r}_0$ strongly in $H^{-1}$, as $\epsilon\to0$. Thanks to Proposition \ref{weak_solution_app},  for each $\epsilon>0$   there exists a  solution $\rho_{\epsilon} \in\,L^{2}(\Omega;C([0,T];H)\cap L^{2}([0,T];H^{1}))$ for problem \eqref{limit_para} with initial condition $\mathfrak{r}_{\epsilon}$. If for every $\epsilon,\delta>0$ we define 
\[\vartheta_{\epsilon, \delta}(t):=\rho_{\epsilon}(t)-\rho_{\delta}(t),\ \ \ \ \ \ t \in\,[0,T],\] we have
\begin{equation*}
	\begin{array}{l}
		\ds{\frac{1}{2}\,d\norm{\vartheta_{\epsilon, \delta}(t)}_{H^{-1}}^{2}}\\[14pt]
		\ds{=-\Inner{B(\rho_{\epsilon}(t))-B(\rho_{\delta}(t)),\vartheta_{\epsilon, \delta}(t)}_{H}dt +\Inner{F_{g}(\rho_{\epsilon}(t))-F_{g}(\rho_{\delta}(t)),\vartheta_{\epsilon, \delta}(t) }_{H^{-1}}dt }\\[14pt]
		\ds{\quad +\frac{1}{2}\norm{\sigma_{g}(\rho_{\epsilon}(t))-\sigma_{g}(\rho_{\delta}(t))}_{\mathcal{L}_{2}(H_{Q},H^{-1})}^{2}dt +\Inner{\vartheta_{\epsilon, \delta}(t),\big[\sigma_{g}(\rho_{\epsilon}(t))-\sigma_{g}(\rho_{\delta}(t))\big]dw^{Q}(t)}_{H^{-1}}. }
		\end{array}
\end{equation*} 
Since
\[B(r)=\int_0^r b(s)\,ds=\int_0^r\frac{1}{\gamma(g^{-1}(s))}\,ds,\]
we have
\[( B(r_1)-B(r_2))(r_1-r_2)\geq \frac 1{\gamma_1}\,|r_1-r_2|^2,\ \ \ \ \ \ r_1, r_2 \in\,\mathbb{R},\]
so that
\begin{equation}\label{sm30}
	\begin{array}{l}
		\ds{\frac{1}{2}\,d\norm{\vartheta_{\epsilon, \delta}(t)}_{H^{-1}}^{2}\leq -c_0 \norm{\vartheta_{\epsilon, \delta}(t)}_{H}^{2}dt+ \Inner{\vartheta_{\epsilon, \delta}(t),\big[\sigma_{g}(\rho_{\epsilon}(t))-\sigma_{g}(\rho_{\delta}(t))\big]dw^{Q}(t)}_{H^{-1}}, }
	\end{array}\end{equation}
		where
		\[c_0
		:=\left(\frac{1}{\gamma_{1}}-\frac{L_\sigma}{2\alpha_{1}\gamma_{0}^{2}}-\frac{L_f}{\alpha_{1}\gamma_{0}}\right)>0,\]
last inequality following from \eqref{additional_condition}.

Hence, if we  first integrate both sides in \eqref{sm30} with respect to time and then   take  the expectation, we get
\begin{equation*}
\sup_{s \in\,[0,T]}\mathbb{E}\,\norm{\vartheta_{\epsilon, \delta}(s)}_{H^{-1}}^{2}+2\,c_0\int_0^T \E\norm{\vartheta_{\epsilon, \delta}(s)}_{H}^{2}\,ds\leq \E\,\Vert \mathfrak{r}_{\epsilon}-\mathfrak{r}_{\delta}\Vert_{H^{-1}}^2
\end{equation*} 
and this implies that the sequence $(\rho_{\epsilon})$ is Cauchy in $C([0,T];L^2(\Omega;H^{-1}))\cap L^2(\Omega;L^{2}(0,T;H)))$. In particular, it converges to some $\rho$ in $C([0,T];L^2(\Omega;H^{-1}))\cap L^2(\Omega;L^{2}(0,T;H)))$, as $\epsilon\to 0$. 
For every $\epsilon>0$ and $\varphi \in\,C^\infty_0(\mathcal{O})$ we have
\[\begin{aligned}
\Inner{\rho_\epsilon(t),\varphi}_{H} 
			&=\Inner{\mathfrak{r}_{\epsilon},\varphi}_{H}+\int_{0}^{t}\!\left(\Inner{B(\rho_\epsilon(s)),\Delta\varphi}_{H}+\Inner{F_{g}(\rho_\epsilon(s)),\varphi}_{H}\right)ds+\int_{0}^{t}\!\Inner{\varphi,\sigma_{g}(\rho_\epsilon(s))dw^{Q}(s)}_{H},	\end{aligned}
\]
Then, due to the Lipschitz continuity of $B$, $F_{g}$ and $\sigma_{g}$, we can take the limit in both sides of the identity above, as $\epsilon\to 0$, and we get  that $\rho$ is a  solution for \eqref{limit_para}.

To prove the uniqueness, assume that $\rho_{1},\rho_{2}$ are two solutions to \eqref{limit_para}. By proceeding as above, we have
\begin{equation*}
	\sup_{t \in\,[0,T]}\E\,\norm{\rho_{1}(t)-\rho_{2}(t)}_{H^{-1}}^{2}+c_0\int_0^T\E\,\norm{\rho_{1}(t)-\rho_{2}(t)}_{H}^{2}\,dt\leq0,
\end{equation*} 
which gives $\rho_{1}=\rho_{2}$. 

Next, we prove that $\rho \in\,L^2(\Omega;C([0,T];H^{-1})\cap L^{2}(0,T;H)))$. We apply  It\^o's formula to $\norm{\rho}_{H^{-1}}^{2}$ and we get
	\begin{equation*}
		\begin{array}{l}
			\ds{\frac{1}{2}d\norm{\rho(t)}_{H^{-1}}^{2} }
			\ds{= \frac{1}{2}\norm{\sigma_{g}(\rho(t))}_{\mathcal{L}_{2}(H_{Q},H^{-1})}^{2}dt-\inner{ B(\rho(t)),\rho(t)}_{H}dt}\\[14pt]
			\ds{\hsp\hsp+\Inner{F_{g}(\rho(t)),\rho(t)}_{H^{-1}}dt+\inner{\rho(t),\sigma_{g}(\rho(t))dw^{Q}(t)}_{H^{-1}}}\\[14pt]
			\ds{\hsp \quad\leq c-\gamma_{1}^{-1}\norm{\rho(t)}_{H}^{2}dt+\Inner{F_{g}(\rho(t)),\rho(t)}_{H^{-1}}dt+\inner{\rho(t),\sigma_{g}(\rho(t))dw^{Q}(t)}_{H^{-1}}}.
		\end{array}
	\end{equation*}
Due to \eqref{L_F_condition} there exists $\bar{\delta}>0$ such that
\[c_1:=\frac 1{\gamma_1}-\frac{L_f+\bar{\delta}}{\alpha_1 \gamma_0}>0,\]
so that, thanks to 
\eqref{Fg3} we have
\begin{equation}\label{sm-tris1}
		\begin{array}{l}
			\ds{\frac{1}{2}d\norm{\rho(t)}_{H^{-1}}^{2} \leq c-c_1\norm{\rho(t)}_{H}^{2}dt+\inner{\rho(t),\sigma_{g}(\rho(t))dw^{Q}(t)}_{H^{-1}}}.
		\end{array}
	\end{equation}
Now, since we have
\[\E\,\sup_{s \in\,[0,t]}\left\vert\int_0^s \inner{\rho(t),\sigma_{g}(\rho(t))dw^{Q}(t)}_{H^{-1}}\right\vert\leq \frac 14 \E\,\sup_{s \in\,[0,t]}\Vert \rho(s)\Vert_{H^{-1}}^2+c,\]
if we integrate both sides in \eqref{sm-tris1} and then take the supremum with respect to time and the expectation, we get
\[\E\,\sup_{t \in\,[0,T]}\Vert \rho(t)\Vert^2_{H^{-1}}+\int_0^T \E\,\Vert \rho(s)\Vert_{H}^2\,ds\leq c_T\left(1+\Vert \mathfrak{r}_0\Vert_{H^{-1}}^2\right),\]
which, in particular implies that $\rho \in\,L^2(\Omega;L^\infty(0,T;H^{-1})\cap L^{2}(0,T;H)))$. Moreover, since $\rho$ solves equation \eqref{limit_para}, it belongs to $C([0,T];H)$, $\mathbb{P}$-a.s.

Finally, in order to prove   \eqref{lim_ene1}, we take the expectation of both sides of \eqref{sm-tris1} and we get
\[\frac d{dt}\E\,\Vert\rho(t)\Vert_{H^{-1}}^2+2\,c_1 \E\,\Vert\rho(s)\Vert_{H}^2\,ds\leq c\]
and  this implies that there exist some $c, \lambda>0$ such that \eqref{lim_ene1} holds.
\end{proof}

%%%%%%%%%%%%%%%%%%%%%%%%%%%%%%%%%%%%%%%%%%%%%%%%%%%%%%%%%

\section{Ergodic behavior of the limiting equation}

In what follows, we denote by $R^{H^{-1}}_{t}$ the transition semigroup associated to equation \eqref{limit_para} on $H^{-1}$
\begin{equation*}
	R^{H^{-1}}_{t}\varphi(\mathfrak{r}):=\E\varphi(\rho^{\mathfrak{r}}(t)),\ \ \ \ \ \ \ \mathfrak{r}\in H^{-1},\ \ \ \ \ \ \ t\geq0,
\end{equation*}
for every $\varphi\in B_{b}(H^{-1})$. 
Similarly, we denote by $R^{H}_{t}$ the transition semigroup associated to equation \eqref{limit_para} on $H$,
\begin{equation*}
	R^{H}_{t}\varphi(\mathfrak{r}):=\E\varphi(\rho^{\mathfrak{r}}(t)),\ \ \ \mathfrak{r}\in H,\ \ \ \ t\geq0,
\end{equation*}
for every $\varphi\in B_{b}(H)$. Clearly, if $\mathfrak{r}\in H$ and $\varphi\in B_{b}(H^{-1})$, then 
\begin{equation*}
	R_{t}^{H}\varphi(\mathfrak{r})=R^{H^{-1}}_{t}\varphi(\mathfrak{r}),\ \ \ \ \ \ t\geq 0.
\end{equation*}

For every $A \in\,\mathcal{B}(H^{-1})$ we have that $A\cap H \in\,\mathcal{B}(H)$. Thus,  if $\nu\in\mathcal{P}(H)$, we can define its extension $\nu'\in\mathcal{P}(H^{-1})$ by setting
\begin{equation*}
	\nu'(A)=\nu(A\cap H),\ \ \ \ \ \ \ \ \   A\in \mathcal{B}(H^{-1}).
\end{equation*}
With this definition, $\text{supp}\,(\nu')\subset H$. Indeed, if we denote by $B_{H}(\mathfrak{r},R)$ the closed ball in $H$ centered at $\mathfrak{r}\in H$ with radius $R>0$, then $B_{H}^{c}(\mathfrak{r},R)\in\mathcal{B}(H)$, so that
\begin{equation*}
	\lim_{R\to+\infty}\nu'(B_{H}^{c}(0,R))=\lim_{R\to+\infty}\nu(B_{H}^{c}(0,R))=0,
\end{equation*}
which implies that $\text{supp}\,(\nu')\subset H$.

\begin{Proposition}\label{lim_invariant}
 Assume Hypotheses \ref{Hypothesis1}, \ref{Hypothesis2}, \ref{Hypothesis3} and \ref{Hypothesis_control}, and define
 $\alpha(\mathfrak{r},\mathfrak{s}):=\abs{\mathfrak{r}-\mathfrak{s}}_{H^{-1}}$.
Then,  there   exist some positive constant $\lambda_{0},t_{0}$ and $c$ such that
\begin{equation}\label{contraction}
	\mathcal{W}_{\alpha}\left((R^{H^{-1}}_{t})^{\ast}\nu_{1},(R^{H^{-1}}_{t})^{\ast}\nu_{2}\right)\leq c\,e^{-\lambda_{0}t}\,\mathcal{W}_{\alpha}(\nu_{1},\nu_{2}),\ \ \ \ \ \   t>t_{0}.
\end{equation}
Moreover,  $R^{H^{-1}}_{t}$ has a unique invariant measure $\nu^{H^{-1}}$ such that $\text{\em{supp}}\,(\nu^{H^{-1}})\subset H^{1}$ and
\begin{equation}\label{speed}
	\mathcal{W}_{\alpha}\left((R^{H^{-1}}_{t})^\ast \delta_\mathfrak{r},\,\nu^{H^{-1}}\right)\leq c\left(1+\norm{\mathfrak{r}}_{H^{-1}}\right)\,e^{-\lambda_{0}t},\ \ \ \ \ \ t\geq0,\ \ \ \ \mathfrak{r}\in H^{-1}.
\end{equation}
\end{Proposition}

\begin{proof}
	Let $\rho^{\mathfrak{r}_1},\rho^{\mathfrak{r}_2}$ be two solutions of \eqref{limit_para}, with initial conditions  $\mathfrak{r}_1, \mathfrak{r}_2 \in\,H^{-1}$, respectively. By proceeding as in   the proof of Proposition \ref{weak_solution}, we have
	\begin{equation*}
		\E\norm{\rho^{\mathfrak{r}_1}(t)-\rho^{\mathfrak{r}_1}(t)}_{H^{-1}}^{2}\leq e^{-\lambda t}\norm{\mathfrak{r}_1-\mathfrak{r}_2}_{H^{-1}}^{2},\ \ \ \ \ \ t\geq 0,
	\end{equation*}
	for some constant $\lambda>0$. 
	In particular,  the semigroup $R_t^{H^{-1}}$ is Feller in $H^{-1}$ and  for every $\varphi\in\text{Lip}_{b}(H^{-1})$ and $\mathfrak{r}_1, \mathfrak{r}_2\in H^{-1}$ 
	\begin{equation}\label{sm9}
		\left\lvert R_t^{H^{-1}}\varphi(\mathfrak{r}_1)-R_t^{H^{-1}}\varphi(\mathfrak{r}_2)\right\rvert\leq [\varphi]_{\text{Lip}_{H^{-1},\alpha}}e^{-\lambda t/2}\norm{\mathfrak{r}_1-\mathfrak{r}_2}_{H^{-1}},\ \ \ \ \ \ \ \ \   t\geq 0.
	\end{equation}
As shown e.g. in \cite[Theorem 2.5]{HM}, \eqref{sm9} implies \eqref{contraction}. Moreover, it implies that $R_t^{H^{-1}}$ has at most one invariant measure.	
	
If for every $R>0$ and $t>0$ we denote	\begin{equation*}
		B_{R}:=\left\{\mathfrak{r}\in H^{-1}:\norm{\mathfrak{r}}_{H^{1}}\leq R\right\},\ \ \ \ \ \ \ 
		\Gamma_{t}:=\frac{1}{t}\int_{0}^{t}(R_t^{H^{-1}})^\ast\,\delta_0\,dt.
	\end{equation*}
Then, thanks to  \eqref{lim_ene3}, for every $R>0$  and $t>0$ we have
	\begin{equation}\label{sm44}
		R_{t}(B_{R}^{c})=\frac{1}{t}\int_{0}^{t}\mathbb{P}\left(\norm{\rho^0(s)}_{H^{1}}> R\right)ds\leq \frac{c}{R^{2}}.
	\end{equation}
	Since $B_R$ is compactly embedded in $H^{-1}$, this implies   that the family of measures $\{\Gamma_t\}_{t>0}$, is tight in $H^{-1}$. Then, by  Prokhorov's Theorem,  there exists $t_n\uparrow \infty$ such that  $\Gamma_{t_n}$ converges weakly to some probability measure in $\mathcal{P}(H^{-1})$ which is invariant for $R^{H^{-1}}_t$ and, due to what we have seen above, such measure is the unique invariant measure  $\nu^{H^{-1}}$ of $R^{H^{-1}}_t$. Moreover, \eqref{sm44} gives
	\begin{equation*}
		\nu^{H^{-1}}(B_{R}^{c})\leq \liminf_{n\to\infty}\Gamma_{t_n}(B^c_R)\leq \frac{c}{R^{2}},\ \ \ \ \   R>0,
	\end{equation*}
	so that  $\text{supp}\,(\nu^{H^{-1}})\subset H^{1}$.
	
	Finally, in order to  prove \eqref{speed}, we first notice that due to the invariance of $\nu^{H^{-1}}$ and \eqref{lim_ene1}	\begin{equation*}\label{sm10}
	\begin{aligned}
	\int_{H^{-1}}&\Vert \mathfrak{r}\Vert^2_{H^{-1}}\,\nu^{H^{-1}}(d\mathfrak{r})	\leq \liminf_{R\to\infty}\int_{H^{-1}}\left(\Vert \mathfrak{r}\Vert^2_{H^{-1}}\wedge R\right)\,\nu^{H^{-1}}(d\mathfrak{r})\\[10pt]
	&=\liminf_{R\to\infty}\int_{H^{-1}}\left(\E\,\Vert \rho^{\mathfrak{r}}(t)\Vert^2_{H^{-1}}\wedge R\right)\,\nu^{H^{-1}}(d\mathfrak{r})\leq c\left(1+e^{-\la t}\int_{H^{-1}}\Vert \mathfrak{r}\Vert^2_{H^{-1}}\nu^{H^{-1}}(d\mathfrak{r})\right).
	\end{aligned}
	\end{equation*}
	Thus, if we take $\bar{t}>0$ such that $c e^{-\lambda \bar{t}}=1/2$, we get 
	\begin{equation}
	\label{sm45}
	\int_{H^{-1}}\Vert \mathfrak{r}\Vert^2_{H^{-1}}\,\nu^{H^{-1}}(d\mathfrak{r})	\leq c.	
	\end{equation}
	Then, in view of \eqref{sm9}, for every $\varphi \in\,\text{Lip}_b(H^{-1})$ we have
	\[\begin{aligned}
\mathcal{W}_{\alpha}&\left((R^{H^{-1}}_{t})^\ast \delta_\mathfrak{r},\,\nu^{H^{-1}}\right)	\leq \left\vert R^{H^{-1}}_t\varphi(\mathfrak{r})-\int_{H^{-1}}\varphi(\mathfrak{s})\,\nu^{H^{-1}}(d\mathfrak{s})\right\vert\\[10pt]
&\leq \int_{H^{-1}}\left\vert R^{H^{-1}}_t\varphi(\mathfrak{r})-R^{H^{-1}}_t\varphi(\mathfrak{s})\right\vert \nu^{H^{-1}}(d\mathfrak{s})\leq [\varphi]_{\tiny{\text{Lip}_{H^{-1}}, \alpha}}\, e^{-\lambda t/ 2}\int_{H^{-1}}\Vert \mathfrak{r}-\mathfrak{s}\Vert_{H^{-1}}\nu^{H^{-1}}(d\mathfrak{s}),
\end{aligned}\]
and \eqref{sm45} allows to obtain \eqref{speed}, with $\lambda_0=\lambda/2$.

\end{proof}

\begin{Remark}
{\em  Based on the fact that $\mathcal{B}(H)\subset \mathcal{B}(H^{-1})$ and the fact that $\text{supp}\,(\nu^{H^{-1}})\subset H^{1}$, we have that $\nu^{H^{-1}}\in\mathcal{P}(H^{-1})$ is also a probability measure on $H$. In what follows, it will be convenient to distinguish the restriction of $\nu^{H^{-1}}$ to $H$ from $\nu^{H^{-1}}$ itself and for this reason we will denote it by $\nu^{H}$.
 	}
\end{Remark}

\begin{Proposition}\label{invariant_H}
	The probability measure $\nu^{H}$ is the unique invariant measure for the transition semigroup $R_t^{H}$.	Moreover, $\text{{\em supp}}\,(\nu^{H})\subset H^{1}$ and
	\begin{equation}\label{uniform_limit}
		\int_{H}\norm{\mathfrak{r}}_{H^{1}}^{2} \nu^{H}(d\mathfrak{r})<\infty.
	\end{equation}
\end{Proposition}

\begin{proof}
	By proceeding as in the proof of Lemma \ref{sys_invariant_H1}, it is possible to show that $\nu^{H}$ is invariant for $R_{t}^{H}$, and from Proposition \ref{lim_invariant} we get that  $\text{supp}(\nu^{H})\subset H^{1}$.
		
	To prove its uniqueness, we notice that if $\nu\in \mathcal{P}(H)$ is any invariant measure for $R_{t}^{H}$, then its extension $\nu'\in\mathcal{P}(H^{-1})$, with the support in $H$, is invariant for $R_t^{H^{-1}}$. From Proposition \ref{lim_invariant}, we have $\nu'=\nu^{H^{-1}}$, and hence 
	\begin{equation*}
\nu(A)=\nu'(A)=\nu^{H^{-1}}(A)=\nu^{H}(A),\ \ \ \ \ \   A\in\mathcal{B}(H),
	\end{equation*}
	which implies that $\nu=\nu^{H}$.
	
	Finally, 
	in order to prove \eqref{uniform_limit}, we consider the Komolgov operator associated to $R_{t}^{H}$
	\begin{equation*}
		N\varphi(\mathfrak{r})
		=\frac{1}{2}\text{Tr}_{H}\Big[\big(\sigma_{g}(\mathfrak{r})Q\big)\big(\sigma_{g}(\mathfrak{r})Q\big)^{\ast}D^{2}\varphi(\mathfrak{r})\Big]+\Inner{\text{div}\big(b(\mathfrak{r})\nabla \mathfrak{r}\big)+F_{g}(\mathfrak{r}),D\varphi(\mathfrak{r}) }_{H},
	\end{equation*}
	We consider the function $\varphi(\mathfrak{r}):=\norm{\mathfrak{r}}_{H}^{2}/2$,
	then 
	\begin{equation*}
		\begin{array}{ll}
			\ds{N\varphi(\mathfrak{r}) }
			\ds{=\frac{1}{2}\norm{\sigma_{g}(\mathfrak{r})}_{\mathcal{L}_{2}(H_{Q},H)}^{2}-\Inner{ b(\mathfrak{r})\nabla \mathfrak{r}, \nabla \mathfrak{r}}_{H}+\Inner{F_{g}(\mathfrak{r}),\mathfrak{r}}_{H}\leq \frac{1}{2}\sigma_{\infty}^{2}-\gamma_{1}^{-1}\norm{\nabla \mathfrak{r}}_{H}^{2}+\Inner{F_{g}(\mathfrak{r}),\mathfrak{r}}_{H} },
		\end{array}
	\end{equation*}
	so thanks to \eqref{Fg2} and \eqref{L_F_condition}, by the invariance of $\nu^{H}$ on $H$ we have
	\begin{equation*}
		\int_{H}\norm{\mathfrak{r}}_{H^{1}}^{2}\nu^{H}(d\mathfrak{r})<\infty.
	\end{equation*}
\end{proof}

\begin{Remark}
	{\em As a direct consequence of \eqref{uniform_limit}, we have 
	\begin{equation}\label{uniform_limit1}
		\int_{H^{-1}}\norm{\mathfrak{r}}_{H^{1}}^{2}\nu^{H^{-1}}(d\mathfrak{r})<\infty.
	\end{equation}}
\end{Remark}

%%%%%%%%%%%%%%%%%%%%%%%%%%%%%%%%%%%%%%%%%%%%%%%%%%%%%%%%%%

\section{Proof of Theorem \ref{main-teo}}

 Due to Hypothesis \ref{Hypothesis2}, with an abuse of notation in this section we will look at  $g$ and $g^{-1}$ as mappings on $H$
\begin{equation*}
	[g(h)](x):=g(h(x)),\ \ \ \ \ [g^{-1}(h)](x):=g^{-1}(h(x)),\ \ \ x\in\mathcal{O},\ \ \ \ \ \ \ h\in H.
\end{equation*} 
For every probability measure $\nu\in\mathcal{P}(H)$, we define probability measures $\nu\circ g$ and  $\nu\circ g^{-1}$ $\in \mathcal{P}(H)$ by
\begin{equation*}
	\big(\nu\circ g\big)(A):=\nu(g(A)),\ \ \ \big(\nu\circ g^{-1}\big)(A):=\nu(g^{-1}(A)),\ \ \   A\in \mathcal{B}(H).
\end{equation*}
Clearly, we have
\begin{equation}\label{coincide2}
	(\nu\circ g)\circ g^{-1}=(\nu\circ g^{-1})\circ g=\nu.
\end{equation}

Now, we denote by ${P}^{H}_{t}$ the transition semigroup associated to the limiting problem \eqref{limiting_problem}
\begin{equation*}
	P^H_t\varphi(\mathfrak{u}):=\E\,\varphi(u^{\mathfrak{u}}(t)),\ \ \ \ \ \ \ \mathfrak{u}\in H,\ \ \ \ \ \ t\geq0,
\end{equation*}
for every $\varphi\in B_{b}(H)$. For every $\mathfrak{r}, \mathfrak{u} \in\,H$ and $t\geq 0$ we have 
\[g^{-1}(\rho^{\mathfrak{r}}(t))=u^{g^{-1}(\mathfrak{r})}(t),\ \ \ \ \ \ \ \ \ \ \rho^{g(\mathfrak{u})}(t)=u^{\mathfrak{u}}(t).\]
Hence,
if we define the operator $T_{g}:C_{b}(H)\to C_{b}(H)$ by
\[
	[T_{g}\varphi](\mathfrak{u})=\varphi(g(\mathfrak{u})),\ \ \ \mathfrak{u}\in H,
\]
we have $T_{g}^{-1}:=T_{g^{-1}}$, 
\begin{equation}
\label{sm15}
\int_{H} [T_g\varphi	](\mathfrak{u})\,(\nu\circ g)(d\mathfrak{u})=\int_H\varphi(\mathfrak{r})\,\nu(d\mathfrak{r}),
\end{equation}
and for every $\varphi \in\,C_b(H)$
\begin{equation}
\label{connection}	
R^H_t\varphi(\mathfrak{r})=\E\,\varphi(\rho^{\mathfrak{r}}(t))=\E\,[T_{g}\varphi](u^{g^{-1}(\mathfrak{r})}(t))=P^H_t[T_{g}\varphi](g^{-1}(\mathfrak{r})),\ \ \ \ \ \   t\geq0.
\end{equation}

\begin{Lemma}\label{l1sm}
$\nu \in\,\mathcal{P}(H)$  is invariant for $P^H_t$ if and only $\nu\circ g^{-1}$ is invariant for $R^H_t$.	 In particular, $\nu^H\circ g$ is the  unique invariant measure for $P^H_t$.
\end{Lemma}
\begin{proof}
Assume $\nu \in\,\mathcal{P}(H)$ is invariant for $P^H_t$. Then, thanks to \eqref{connection} and \eqref{sm15}, for every $\varphi \in\,C_b(H)$ and $t\geq 0$ we have
\[\begin{aligned}
\int_H& R^H_t\varphi(\mathfrak{r})\,(\nu\circ g^{-1})(d\mathfrak{r})=	\int_H R^H_t\varphi(g(\mathfrak{u}))\,\nu(d\mathfrak{u})=	\int_H P^H_t[T_g\varphi](\mathfrak{u})\,\nu(d\mathfrak{u})\\[10pt]
&\hsp=\int_H [T_g\varphi](\mathfrak{u})\,\nu(d\mathfrak{u})=\int_H\varphi(\mathfrak{r})\,(\nu\circ g^{-1})(d\mathfrak{r}).
\end{aligned}\]
This implies that $\nu\circ g^{-1}$ is invariant for $R^H_t$. In the same way, if $\lambda \in\,\mathcal{P}(H)$ is invariant for $R^H_t$, then $\lambda\circ g$ is invariant for $P^H_t$. Hence, we can conclude due to \eqref{coincide2}. 

Our statement can be rephrased by saying that  there exists a unique invariant measure for $R^H_t$ if and only is there exists a unique invariant measure for $P^H_t$. Therefore, since we have shown in Corollary \ref{invariant_H} that $\nu^H$ is the unique invariant measure for $R^H_t$, we obtain that $\nu^H\circ g$ is the unique invariant measure for $P^H_t$.

 	\end{proof}

Now, we can conclude that Theorem \ref{main-teo} holds, once we prove that if
 $(\nu_{\mu}^{ \mathcal{H}})_{\mu>0}\subset\mathcal{P}(\H)$ is a family of invariant measures  for the transition semigroups $ P_t^{\mu, \mathcal{H}}$, such that $\text{supp}(\nu_{\mu}^{ \mathcal{H}})\subset \H_{1}$,  then  we have
	\begin{equation}\label{main_convergence}
		\lim_{\mu\to0}\mathcal{W}_{\alpha}\left(\left[\left(\Pi_{1}\nu_{\mu}^{ \mathcal{H}}\right)\circ g^{-1}\right]',\nu^{H^{-1}}\right)=0,
	\end{equation}
	where  $\nu^{H^{-1}}$ is the unique invariant measure for $R_t^{H^{-1}}$ in $H^{-1}$.

 Actually, 
in view of \eqref{uniform_invariant},  the family of probability measures $(\Pi_{1}\nu_{\mu}^{ \mathcal{H}})_{\mu\in(0,1)}$ is tight in $H^{\delta}$, for every $\delta<1$. If $\nu$ is any weak limit of $\Pi_{1}\nu_{\mu}^{ \mathcal{H}}$ in $H$,  as $\mu\to0$, we have $(\Pi_{1}\nu_{\mu}^{ \mathcal{H}})\circ g^{-1}$ converges weakly to $\nu\circ g^{-1}$ on $H$. Due to the continuity of the embedding of $H^{-1}$ into $H$,
\begin{equation*}
	\Big[(\Pi_{1}\nu_{\mu}^{ \mathcal{H}})\circ g^{-1}\Big]'\rightharpoonup\big(\nu\circ g^{-1}\big)',\ \ \ \ \ \text{as}\ \mu\to 0, 
\end{equation*} 
as measures on $H^{-1}$. On the other hand, according to \eqref{main_convergence} we have that $\left[(\Pi_{1}\nu_{\mu}^{ \mathcal{H}})\circ g^{-1}\right]'$ converges weakly to $\nu^{H^{-1}}$ in $H^{-1}$, so that $(\nu\circ g^{-1})'=\nu^{H^{-1}}$ in $H^{-1}$. This implies that $\nu\circ g^{-1}=\nu^{H}\in\mathcal{P}(H)$, and thus $\nu=\nu^{H}\circ g\in\mathcal{P}(H)$. Since this holds for every weak limit $\nu$ of $\Pi_{1}\nu_{\mu}^{ \mathcal{H}}$, we conclude that $\Pi_{1}\nu_{\mu}^{ \mathcal{H}}$ converges weakly to $\nu^{H}\circ g$ in $H$, as $\mu\to0$, and, due to Lemma \ref{l1sm},  $\nu^{H}\circ g$ is the unique invariant measure for $P_{t}^{H}$.

\subsection{Proof of \eqref{main_convergence}}
Due to the invariance of $\nu_\mu^{\mathcal{H}}$ and $\nu^{H^{-1}}$, we have
\[\begin{aligned}
\mathcal{W}_{\alpha}&\left(\left[\left(\Pi_{1}\nu_{\mu}^{ \mathcal{H}}\right)\circ g^{-1}\right]',\nu^{H^{-1}}\right)\leq \mathcal{W}_{\alpha}\left(\left[\Pi_1((P_t^{\mu, \mathcal{H}})^\star \nu_\mu^{\mathcal{H}})\circ g^{-1}\right]^\prime,(R^{H^{-1}}_{t})^{\ast}\left[\left(\Pi_{1}\nu_{\mu}^{ \mathcal{H}}\right)\circ g^{-1}\right]'\r)\\[10pt]
&\hsp+\mathcal{W}_{\alpha}\left(
(R^{H^{-1}}_{t})^{\ast}\left[\left(\Pi_{1}\nu_{\mu}^{ \mathcal{H}}\right)\circ g^{-1}\right]',(R^{H^{-1}}_{t})^{\ast}\nu^{H^{-1}}
\right).	
\end{aligned}
\]
According to \eqref{contraction}, we have
\[\mathcal{W}_{\alpha}\left(
(R^{H^{-1}}_{t})^{\ast}\left[\left(\Pi_{1}\nu_{\mu}^{ \mathcal{H}}\right)\circ g^{-1}\right]',(R^{H^{-1}}_{t})^{\ast}\nu^{H^{-1}}
\right)\leq c\, e^{-\lambda_0 t}\mathcal{W}_{\alpha}\left(
\left[\left(\Pi_{1}\nu_{\mu}^{ \mathcal{H}}\right)\circ g^{-1}\right]',\nu^{H^{-1}}
\right),\]
and then, if we pick $\bar{t}>0$ such that $c e^{-\lambda_0 \bar{t}}\leq 1/2$, we obtain
\[\mathcal{W}_{\alpha}\left(\left[\left(\Pi_{1}\nu_{\mu}^{ \mathcal{H}}\right)\circ g^{-1}\right]',\nu^{H^{-1}}\right)\leq 2\,\mathcal{W}_{\alpha}\le(\left[(\Pi_1((P_{\bar{t}}^{\mu, \mathcal{H}})^\star \nu_\mu^{\mathcal{H}})\circ g^{-1}\right]^\prime,(R^{H^{-1}}_{\bar{t}})^{\ast}\left[\left(\Pi_{1}\nu_{\mu}^{ \mathcal{H}}\right)\circ g^{-1}\right]'\r).\]
Now, if we fix a $\mathcal{F}_{0}$-measurable $\mathcal{H}_{1}$-valued random variable $\vartheta_\mu:=(\xi_\mu,\eta_\mu)$, distributed as the invariant measure $\nu_{\mu}^{ \mathcal{H}}$, the Kantorovich-Rubinstein identity \eqref{KR-id} 
gives for every $t\geq 0$
\[\mathcal{W}_{\alpha}\le(\left[(\Pi_1((P_t^{\mu, \mathcal{H}})^\star \nu_\mu^{\mathcal{H}})\circ g^{-1}\right]^\prime,(R^{H^{-1}}_{t})^{\ast}\left[\left(\Pi_{1}\nu_{\mu}^{ \mathcal{H}}\right)\circ g^{-1}\right]'\r)\leq \mathbb{E}\,\alpha (g(u_\mu^{\vartheta_\mu}(t)), \rho^{\,g(\xi_\mu)}(t)).\] 
Thus,   \eqref{main_convergence} follows once we prove that for every $t\geq 0$ large enough
	\begin{equation}
	\label{6.7}
\lim_{\mu\to 0}\,\mathbb{E}\,\alpha (g(u_\mu^{\vartheta_\mu}(t)), \rho^{\,g(\xi_\mu)}(t))=\lim_{\mu\to 0}\E\,\Vert g(u_\mu^{\vartheta_\mu}(t))-\rho^{\,g(\xi_\mu)}(t)\Vert_{H^{-1}}=0.
	\end{equation}
	
According to \eqref{uniform_invariant} we have that $\vartheta_\mu\in L^{2}(\Omega;\H_{1})$, for every $\mu\in(0,1)$. Hence, if we denote $\rho_\mu(t):=g(u_\mu^{\vartheta_\mu}(t))$, by proceeding  as in \cite[Section 5]{cerraixi}, we can rewrite equation  \eqref{system} in the following way 
\begin{equation*}
	\begin{array}{ll}
		\ds{\rho_\mu(t)+\mu\, v_{\mu}^{\vartheta_\mu}(t)}
		\ds{=g(\xi_{\mu})+\mu\, \eta_{\mu}+\int_{0}^{t}\Delta [B(\rho_\mu(s))]ds+\int_{0}^{t}F_{g}(\rho_\mu(s))ds+\int_{0}^{t}\sigma_{g}(\rho_\mu(s))dw^{Q}(s) },
	\end{array}
\end{equation*}
where the identity holds in $H^{-1}$ sense.
Since $\rho^{\,g(\xi_\mu)}$ solves equation \eqref{limit_para} with initial condition $g(\xi_\mu)\in L^{2}(\Omega;H)$ in $H^{-1}$ sense, we have 
\begin{equation*}
	\begin{array}{l}
	\ds{\rho_\mu(t)-\rho^{\,g(\xi_{\mu})}(t)+\mu\, v_{\mu}^{\vartheta_\mu}(t)=\mu\, \eta_{\mu}+\int_{0}^{t}\Delta\left[B(\rho_\mu(s))-B(\rho^{\,g(\xi_{\mu})}(s))\right]\,ds}\\[14pt]
	\ds{ \hsp+\int_{0}^{t}\left(F_{g}(\rho_\mu(s))-F_{g}(\rho^{\,g(\xi_{\mu})}(s))\right)ds+\int_{0}^{t}\left(\sigma_{g}(\rho_\mu(s))-\sigma_{g}(\rho^{\,g(\xi_{\mu})}(s))\right)dw^{Q}(s) }.
	\end{array}
\end{equation*}
If we define $\vartheta_\mu(t):=\rho_\mu(t)-\rho^{\,g(\xi_{\mu})}(t)$, as a consequence of  It\^{o}'s formula, we have
\begin{equation*}
	\begin{array}{l}
		\ds{\frac 12\,\mathbb{E}\,\Vert\vartheta_\mu(t)+\mu\, v_{\mu}^{\vartheta_\mu}(t)\Vert_{H^{-1}}^{2}}\\[14pt]
		\ds{\hslp=\frac 12\,\mu^{2}\,\mathbb{E}\,\Vert\eta_{\mu}\Vert_{H^{-1}}^{2}-\mathbb{E}\,\int_{0}^{t}\Inner{B\big(\rho_\mu(s)\big)-B\big(\rho^{\,g(\xi_{\mu})}(s)\big),\vartheta_\mu(s)+\mu\,v_{\mu}^{\vartheta_\mu}(s) }_{H}ds}\\[14pt]
		\ds{\hsp +\mathbb{E}\,\int_{0}^{t}\Inner{F_{g}\big(\rho_\mu(s)\big)-F_{g}\big(\rho^{\,g(\xi_\mu)}(s)\big),\vartheta_\mu(s)+\mu\,v_{\mu}^{\vartheta_\mu}(s) }_{H^{-1}}ds}\\[14pt]
		\ds{\hsp +\frac 12 \,\mathbb{E}\,\int_{0}^{t}\Vert\sigma_{g}(\rho_\mu(s))-\sigma_{g}(\rho^{\,g(\xi_\mu)}(s))\Vert_{\mathcal{L}_{2}(H_{Q},H^{-1})}^{2}ds}
		\end{array}\end{equation*}
		so that
\begin{equation*}
	\begin{array}{l}
		\ds{\E\,\Vert\vartheta_\mu(t)+\mu\, v_{\mu}^{\vartheta_\mu}(t)\Vert_{H^{-1}}^{2} \leq \mu^{2}\E\norm{\eta_\mu}_{H^{-1}}^{2}-2\mu\,\E\int_{0}^{t}\Inner{B\big(\rho_\mu(s)\big)-B\big(\rho^{\,g(\xi_\mu)}(s)\big),v_{\mu}^{\vartheta_\mu}(s) }_{H}ds }\\[14pt]
		\ds{ +2\mu\,\E\int_{0}^{t}\Inner{F_{g}\big(\rho_\mu(s)\big)-F_{g}\big(\rho^{\,g(\xi_\mu)}(s)\big),v_{\mu}^{\vartheta_\mu}(s) }_{H^{-1}}ds-c_0\,\E\int_{0}^{t}\Vert \vartheta_\mu(s)\Vert_{H}^{2}ds,}
	\end{array}
\end{equation*}
where
\[c_0:=2\Big(\frac{1}{\gamma_{1}}-\frac{L_\sigma}{2\alpha_{1}\gamma_{0}^{2}}-\frac{L_f}{\alpha_{1}\gamma_{0}}\Big)>0.\]
Since $B$ has linear growth, thanks to \eqref{sys_ene1} and \eqref{sys_ene2}  for every $\mu\in(0,\mu_{0})$ we have
\begin{equation*}
	\begin{array}{ll}
		&\ds{\mu\cdot\E\left\lvert\int_{0}^{t}\Inner{B(\rho_\mu(s)\big)-B\big(\rho^{\,g(\xi_\mu)}(s)),v_{\mu}^{\vartheta_\mu}(s) }_{H}ds\right\rvert}\\
		\vs
		&\ds{\leq c\left(\int_{0}^{t}\left(1+\E\,\norm{\rho_\mu(t)}_{H}^{2}+\E\,\Vert\rho^{\,g(\xi_\mu)}(t)\Vert_{H}^{2}\right)dt\right)^{\frac{1}{2}}\left(\int_{0}^{t}\mu^{2}\E\,\Vert v_{\mu}^{\vartheta_\mu}(t)\Vert_{H}^{2}dt\right)^{\frac{1}{2}} }\\
		\vs
		&\ds{ \leq c\left(1+t+\E\norm{\xi_\mu}_{H^{1}}^{2}+\mu^{2}\E\Vert \eta_\mu\Vert_{H}^{2}\right)^{\frac{1}{2}}\left(\mu\,t+\mu^{2}+\mu^{2}\E\Vert\xi_\mu\Vert_{H^{1}}^{2}+\mu^{3}\E\norm{\eta_\mu}_{H}^{2}\right)^{\frac{1}{2}} }\\
		\vs
		&\ds{\leq c_{t}\left(1+\int_{\mathcal{H}}\left(\norm{u}_{H^{1}}^{2}+\mu^{2}\norm{v}_{H}^{2}\right)\nu_{\mu}^{ \mathcal{H}}(du,dv)\right)^{\frac{1}{2}}\left(\mu+\mu^{2}\int_{\mathcal{H}}\left(\norm{u}_{H^{1}}^{2}+\mu\norm{v}_{H}^{2}\right)\nu_{\mu}^{ \mathcal{H}}(du,dv)\right)^{\frac{1}{2}}}\\
		\vs
		&\ds{\hsp\leq c_{t}\sqrt{\mu}\left(1+\int_{\mathcal{H}}\left(\norm{u}_{H^{1}}^{2}+\mu\norm{v}_{H}^{2}\right)\nu_{\mu}^{ \mathcal{H}}(du,dv)\right)}.
	\end{array}
\end{equation*}
Similarly, thanks to the linear growth of $F_{g}$, we have for every $\mu\in(0,\mu_{0})$
\begin{equation*}
\begin{aligned}
\mu\cdot\E&\left\lvert\int_{0}^{t}\Inner{F_{g}(\rho_\mu(s)\big)-F_{g}\big(\rho^{\,g(\xi_\mu)}(s)),v_{\mu}^{\vartheta_\mu}(s) }_{H^{-1}}ds\right\rvert\\[10pt]
&\hsp\leq c_{t}\sqrt{\mu}\left(1+\int_{\mathcal{H}}\left(\norm{u}_{H^{1}}^{2}+\mu\norm{v}_{H}^{2}\right)\nu_{\mu}^{ \mathcal{H}}(du,dv)\right).	
\end{aligned}
\end{equation*}
Moreover, due to \eqref{uniform_invariant} we know the family of random variable $\vartheta_\mu$ satisfies \eqref{initial_condition}. Then, from \eqref{sys_ene3} we obtain that for every $\mu\in(0,\mu_{t})$
\begin{equation*}
	\begin{array}{l}
	\ds{\mu^{2}\E\,\Vert v_{\mu}^{\vartheta_\mu}(t)\Vert_{H^{-1}}^{2}\leq c_{t}\sqrt{\mu}+c\,\mu\left(\E\norm{\xi_\mu}_{H^{1}}^{2}+\mu\,\E\norm{\eta_\mu}_{H}^{2}\right)}\\[14pt]
	\ds{\hsp \hsp=c_{t}\sqrt{\mu}+c\,\mu\int_{\mathcal{H}}\left(\norm{u}_{H^{1}}^{2}+\mu\norm{v}_{H}^{2}\right)\nu_{\mu}^{ \mathcal{H}}(du,dv),}
	\end{array}
\end{equation*}

Therefore, from \eqref{uniform_invariant} and \eqref{additional_condition}  we conclude that for every $\mu \in\,(0,\mu_t)$
\begin{equation*}
	\begin{array}{l}
	\ds{\frac{1}{2}\,\E\,\Vert\rho_\mu(t)-\rho^{\,g(\xi_\mu)}(t)\Vert_{H^{-1}}^{2}\leq \left(\E\,\Vert\rho_\mu(t)-\rho^{\,g(\xi_\mu)}(t)+\mu v_{\mu}^{\vartheta_\mu}(t)\Vert_{H^{-1}}^{2}+\mu^{2}\E\,\Vert v_{\mu}^{\vartheta_\mu}(t)\Vert_{H^{-1}}^{2}\right)}\\[14pt]	\ds{\hsp\leq c_{t}\sqrt{\mu}\left(1+\int_{\mathcal{H}}\left(\norm{u}_{H^{1}}^{2}+\mu\norm{v}_{H}^{2}\right)\nu_{\mu}^{ \mathcal{H}}(du,dv)\right),}
	\end{array}
\end{equation*}
and \eqref{6.7} follows.

%%%%%%%%%%%%%%%%%%%%%%%%%%%%%%%%%%%%%%%%%%%%%

\appendix

\section{Well-posedness of  equation \eqref{limit_para} in $H$} \label{A}

Throughout this section we will not need to assume condition \eqref{L_F_condition} in Hypothesis \ref{Hypothesis3}. Namely, we will just assume 
that	the mapping $f:\mathcal{O}\times\mathbb{R}\to\mathbb{R}$ is measurable, with
	\begin{equation}\label{smapp1}
	\sup_{x\in\mathcal{O}}\abs{f(x,0)}<\infty,\ \ \ \ \ \ \ \ \sup_{x \in\,\mathcal{O}}\abs{f(x,r)-f(x,s)}\leq c\abs{r-s},\ \ \ r,s\in\mathbb{R}.
	\end{equation}
As a consequence of the limiting result proved in \cite{cerraixi}, the well-posedness of equation \eqref{limit_para} has been established when the initial condition $\mathfrak{r}_0\in H^{1}$. Here we want to prove the existence and uniqueness of solutions of \eqref{limit_para} when $\mathfrak{r}_0\in L^{2}(\Omega;H)$.

\begin{Definition}\label{weak_def_app}
	Let $\mathfrak{r}_0\in L^{2}(\Omega;H)$. An adapted process $\rho\in L^{2}(\Omega;C([0,T];H)\cap L^{2}(0,T;H^{1}))$ is a solution of equation \eqref{limit_para} if for every  $\varphi\in C^{\infty}_{0}(\mathcal{O})$
	\begin{equation}\label{weak_test_app}
		\begin{array}{ll}
			\ds{\Inner{\rho(t),\varphi}_{H} }
			&\ds{=\Inner{\mathfrak{r}_0,\varphi}_{H}-\int_{0}^{t}\Inner{b(\rho(s))\nabla\rho(s),\nabla\varphi}_{H}ds }\\
			\vs
			&\ds{\quad+\int_{0}^{t}\Inner{F_{g}(\rho(s)),\varphi}_{H}ds+\int_{0}^{t}\Inner{\varphi,\sigma_{g}(\rho(s))dw^{Q}(s)}_{H},\ \ \ \P\text{-a.s.} }
		\end{array}
	\end{equation}
\end{Definition}

In order to study equation \eqref{limit_para}, we first consider the following approximating problem
\begin{equation}\label{limit_para_approx}
	\le\{\begin{array}{l}
		\ds{\partial_{t}\rho^{\epsilon}(t,x)=\text{div}\big(b(\rho^{\epsilon}(t,x))\nabla\rho^{\epsilon}(t,x)\big)-\epsilon\Delta^{2}\rho^{\epsilon}(t,x)+f(x,\rho(t,x))+\sigma_{g}(\rho(s,\cdot))\partial_{t}w^{Q}(t,x), }\\[10pt]
		\ds{\rho^{\epsilon}(0,x)=\mathfrak{r}_0,\ \ \ \rho^{\epsilon}(t,\cdot)|_{\partial\mathcal{O}}=0, }
	\end{array}\r.
\end{equation}
with $0<\epsilon<<1$ (for a similar approach see e.g. \cite{debussche16}). 

\begin{Lemma}\label{wellposedness_approx}
	Assume Hypotheses \ref{Hypothesis1} and  \ref{Hypothesis2} and condition \eqref{smapp1}. Then, for every $\epsilon, T>0$ and every $\rho_{0}\in L^{2}(\Omega;H)$, equation \eqref{limit_para_approx} admits a unique  solution 
	\begin{equation*}
		\rho_{\epsilon}\in L^{2}(\Omega;C([0,T];H)\cap L^{2}(0,T;H^{2})).
	\end{equation*}
Moreover, there exists some $c_T>0$ such that for every $\epsilon>0$
	\begin{equation}\label{ene_approx}
		\E\sup_{t\in[0,T]}\Vert \rho^{\epsilon}(t)\Vert _{H}^{2}+\frac{2}{\gamma_{1}}\int_{0}^{t}\E\Vert \nabla\rho^{\epsilon}(s)\Vert _{H}^{2}ds+2\,\epsilon\, \int_{0}^{t}\E\Vert \Delta\rho^{\epsilon}(s)\Vert _{H}^{2}ds\leq c_T\left(1+\E\,\Vert \mathfrak{r}_0\Vert _{H}^{2}\right). 
	\end{equation}
	\end{Lemma}

\begin{proof}
	The uniqueness and the existence of  solutions for equation \eqref{limit_para_approx} can be proven by proceeding as in the proof of \cite[Theorem 5.1]{frid}.
		
	In order to prove the energy estimate \eqref{ene_approx}, we apply  It\^o's formula and we get	\begin{equation*}
		\begin{array}{l}
			\ds{\frac 12\,\Vert \rho^{\epsilon}(t)\Vert _{H}^{2} =\frac 12\,\Vert \mathfrak{r}_0\Vert_{H}^{2}+\int_{0}^{t}\Inner{\text{div}\left(b(\rho^{\epsilon}(s))\nabla\rho^{\epsilon}(s)\right),\rho^{\epsilon}(s)}_{H}ds-\epsilon \int_{0}^{t}\Inner{\Delta^{2}\rho^{\epsilon}(s),\rho^{\epsilon}(s)}_{H}ds}\\
			\vs
			\ds{+\int_{0}^{t}\Inner{F_{g}(\rho^{\epsilon}(s)),\rho^{\epsilon}(s)}_{H}\,ds +\frac 12 \int_{0}^{t}\Vert \sigma_{g}(\rho^{\epsilon}(s))\Vert_{\mathcal{L}_2(H_Q,H)}^{2}ds+\int_{0}^{t}\Inner{\rho^{\epsilon}(s),\sigma_{g}(\rho^{\epsilon}(s))dw^Q(s)}_{H}}\\
			\vs
			\ds{\leq \frac 12\,\Vert \mathfrak{r}_0\Vert _{H}^{2}-\frac{1}{\gamma_{1}}\int_{0}^{t}\Vert \nabla\rho^{\epsilon}(s)\Vert _{H}^{2}ds-\epsilon \int_{0}^{t}\Vert \Delta\rho^{\epsilon}(s)\Vert _{H}^{2}ds }\\
			\vs
			\ds{\hsp +c\int_{0}^{t}\Big(1+\Vert \rho^{\epsilon}(s)\Vert _{H}^{2}\Big)ds   + 2\int_{0}^{t}\Inner{\rho^{\epsilon}(s),\sigma_{g}(\rho^{\epsilon}(s))dw^Q(s)}_{H}.}
		\end{array}
	\end{equation*}
	Note that 
	\begin{equation*}
		\begin{array}{ll}
			&\ds{\E\sup_{s\in[0,t]}\left\lvert\int_{0}^{s}\Inner{\rho^{\epsilon}(r),\sigma_{g}(\rho^{\epsilon}(r))dw^Q(r)}_{H}\right\rvert \leq c\,\int_0^t \E\,\Vert \rho^{\epsilon}(s)\Vert _{H}^{2}\,ds+c_T },
		\end{array}
	\end{equation*} 
	and hence
	\begin{equation*}
		\begin{array}{ll}
		&\ds{\E\sup_{s\in[0,t]}\Vert \rho^{\epsilon}(s)\Vert _{H}^{2}+\frac{2}{\gamma_{1}}\E\int_{0}^{t}\Vert \nabla\rho^{\epsilon}(s)\Vert _{H}^{2}ds+2\epsilon\, \E\int_{0}^{t}\Vert \Delta\rho^{\epsilon}(s)\Vert _{H}^{2}ds }\\
		\vs
		&\ds{\hsp\leq \E\Vert \mathfrak{r}_0\Vert _{H}^{2}+c\int_{0}^{t}\E\Vert \rho^{\epsilon}(s)\Vert _{H}^{2}\,ds+c_T. }
		\end{array}
	\end{equation*}
	Therefore, the Gronwall lemma  gives \eqref{ene_approx}.
	
\end{proof}

\begin{Proposition}\label{weak_solution_app}
	Assume Hypotheses \ref{Hypothesis1}, \ref{Hypothesis2} and condition \eqref{smapp1}, and fix $\mathfrak{r}_0\in L^{2}(\Omega;H)$. Then, for every $T>0$, there exists a unique  solution 
	\begin{equation*}
		\rho\in L^{2}(\Omega;C([0,T];H))\cap L^{2}(0,T;H^{1}))
	\end{equation*}
	of equation \eqref{limit_para}. Moreover, there exists some constant $c_T>0$ such that
	\begin{equation}\label{approx_energy}
		\E\sup_{t\in[0,T]}\Vert \rho(t)\Vert _{H}^{2}+\E\int_{0}^{T}\Vert \nabla\rho(s)\Vert _{H}^{2}ds\leq c_T\left(1+\E\Vert \mathfrak{r}_0\Vert _{H}^{2}\right).
	\end{equation}
\end{Proposition}

\begin{proof} By proceeding as in  \cite[Theorem 6.2]{cerraixie}, we can show that equation \eqref{limit_para} admits at most one solution in $L^{2}(\Omega;C([0,T];H))\cap L^{2}(0,T;H^{1}))$. Hence, if we show that there exists a probabilistically weak solution 
\[(\hat{\Omega}, \hat{\mathcal{F}}, \{\hat{\mathcal{F}}\}_t, \hat{\mathbb{P}}, \hat{w}^Q, \hat{\rho} )\] such that  $\hat{\rho} \in\,L^{2}(\hat{\Omega};C([0,T];H))\cap L^{2}(0,T;H^{1}))$, the existence and uniqueness of a probabilistically strong solution for equation \eqref{limit_para} follows.
	
	{\em Step 1.} There exists a filtered probability space $(\hat{\Omega}, \hat{\mathcal{F}}, \{\hat{\mathcal{F}}\}_t, \hat{\mathbb{P}})$, a cylindrical Wiener process $\hat{w}^Q$ associated with $\{\hat{\mathcal{F}}\}_t$ and a process $\hat{\rho}\in L^{2}(\Omega;L^{\infty}(0,T;H))\cap L^{2}(0,T;H^{1}))$ such that
	\begin{equation*}
		\begin{array}{ll}
			\ds{\Inner{\hat{\rho}(t),\varphi}_{H} }
			&\ds{=\Inner{\mathfrak{r}_0,\varphi}_{H}-\int_{0}^{t}\Inner{b(\hat{\rho}(s))\nabla\hat{\rho}(s),\nabla\varphi}_{H}ds }\\
			\vs
			&\ds{\quad+\int_{0}^{t}\Inner{F_{g}(\hat{\rho}(s)),\varphi}_{H}ds+\int_{0}^{t}\Inner{\varphi,\sigma_{g}(\hat{\rho}(s))d\hat{w}^{Q}(s)}_{H},\ \ \ \hat{\P}\text{-a.s.} }
		\end{array}
	\end{equation*}
	for every $\varphi \in\,C^\infty_0(\mathcal{O})$.
	
		{\em Proof of Step 1.}
According to Proposition \ref{wellposedness_approx}, we know that for every $\epsilon>0$ there exists a unique  solution $\rho_{\epsilon}$ to equation \eqref{limit_para_approx}, and
	\begin{equation}\label{tight1}
		\sup_{\e \in\,(0,1)}\left( \E\sup_{t\in[0,T]}\Vert \rho^{\epsilon}(t)\Vert _{H}^{2}+\E\int_{0}^{T}\Vert \rho^{\epsilon}(t)\Vert _{H^{1}}^{2}dt\right)<\infty.
	\end{equation}
For every $ h \in\,(0,T)$ and $t \in\,[0,T-h]$ we have
\[\begin{aligned}
\rho^\epsilon(t+h)&-\rho^\epsilon(t)=\int_t^{t+h} \text{div}\left(b(\rho^\epsilon(s))\nabla \rho^\epsilon(s)\right)\,ds-\epsilon	\int_t^{t+h}\Delta^2 \rho^\epsilon(s)\,ds\\[10pt]
&\hsp+\int_t^{t+h}F_g(\rho^\epsilon(s))\,ds+\int_t^{t+h}\sigma_g(\rho^\epsilon(s))dw^Q(s)=:\sum_{k=1}^4 I^\epsilon_k(t,h).
\end{aligned}\]
We have
\begin{equation}  \label{sm101}\begin{aligned}
\sup_{t \in\,[0,T-h]}\Vert I^\epsilon_1(t,h)\Vert_{H^{-1}}\leq c\,\sup_{t \in\,[0,T]}\int_t^{t+h}\Vert \rho^\epsilon(s)\Vert_{H^{1}}\,ds\leq c\left(\int_0^T	\Vert \rho^\epsilon(s)\Vert_{H^{1}}^2\,ds\right)^{1/2}\,h^{1/2}.
\end{aligned}\end{equation}
For $I^\epsilon_{2}(t,h)$, if $\e \in\,(0,1)$ we have
\begin{equation}  \label{sm102}\begin{aligned}
\sup_{t \in\,[0,T-h]}\Vert I^\epsilon_2(t,h)\Vert_{H^{-3}}\leq \sup_{t \in\,[0,T-h]}\int_t^{t+h}\Vert \rho^\epsilon(s)\Vert_{H^{1}}\,ds\leq c\left(\int_0^T	\Vert \rho^\epsilon(s)\Vert_{H^{1}}^2\,ds\right)^{1/2}\,h^{1/2},
\end{aligned}\end{equation}
and for $I^\epsilon_3(t,h)$ we have
\begin{equation}  \label{sm103}\begin{aligned}
\sup_{t \in\,[0,T-h]}\Vert I^\epsilon_3(t,h)\Vert_{H}\leq c\sup_{t \in\,[0,T-h]}\int_t^{t+h}\left(1+\Vert \rho^\epsilon(s)\Vert_{H}\right)\,ds\leq c_T\left(1+\sup_{t \in\,[0,T]}\Vert \rho^\epsilon(s)\Vert_{H}\right) h,
\end{aligned}\end{equation}
Finally, for $I^\epsilon_4(t,h)$, by using a factorization argument as in \cite[Theorem 5.11 and Theorem 5.15]{DP-Zab}, due to the boundedness of $\sigma_g$ in $\mathcal{L}_2(H_Q,H)$ we obtain that for some $\theta \in\,(0,1)$
\begin{equation}
\label{sm104}
\sup_{\e \in\,(0,1)}\mathbb{E}\,\Vert I^\epsilon_4(t,h)\Vert_{C^\theta([0,T];H)}<\infty.	
\end{equation}
Therefore, by putting together \eqref{sm101}, \eqref{sm102}, \eqref{sm103} and \eqref{sm104}, thanks to \eqref{tight1} we conclude
\[\sup_{t \in\,[0,T-h]}\Vert \rho^\epsilon(t+h)-\rho^\epsilon(t)\Vert_{H^{-3}}\leq c\,h^{\frac 12\wedge \theta},\ \ \ \ \ \ \ h \in\,[0,T),\]
and together with the bound
\[\sup_{\e \in\,(0,1)}\mathbb{E}\,\sup_{t \in\,[0,T]}\Vert \rho^\epsilon(t)\Vert_H<\infty,\]
due to \cite[Theorem 7]{simon} this implies that $\{\rho^\e\}_{\epsilon \in\,(0,1)}$ is tight in $L^\infty(0,T;H^{-\alpha})$, for every $\alpha >0$. Moreover, since for every $\beta \in\,(-3,1)$ we have
\[\Vert u\Vert_{H^\beta}\leq \Vert u\Vert_{H^1}^{\frac{3+\beta}4}\Vert u\Vert_{H^{-3}}^{\frac{1-\beta}4},\]
and the bound
\[\sup_{\e \in\,(0,1)}\int_0^T\mathbb{E}\,\Vert \rho^\epsilon(s)\Vert_{H^1}^2\,ds<\infty\]
holds, 
thanks again to \cite[Theorem 7]{simon}, we have that the family $\{\rho^\e\}_{\epsilon \in\,(0,1)}$ is tight also in the space $L^{8/(3+\beta)}(0,T;H^\beta)$, for every $\beta \in\,(-3,1)$.

In what follows, for every $\alpha>0$ and $\beta \in\,(-3,1)$ we denote
\[\mathcal{X}_{\alpha,\beta}(T):=\left[L^\infty (0,T;H^{-\alpha} )\cap L^{8/(3+\beta)} (0,T;H^{\beta} )\right]\times C([0,T];U),\]
where $U$ is any Hilbert space such that the embedding $H_Q\hookrightarrow U$ is Hilbert-Schmidt. Due to the tightness of $\{\rho^\e, w^Q\}_{\e \in\,(0,1)}$ in $\mathcal{X}_{\alpha, \beta}(T)$, there exists a sequence $\epsilon_n\downarrow0$ such that $\mathcal{L}(\rho^{\e_n}, w^Q)$ is weakly convergent in $\mathcal{X}_{\alpha, \beta}(T)$. Due to  Skorohod's Theorem this implies that there exists a probability space  $(\hat{\Omega}, \hat{\mathcal{F}}, \hat{P})$, a sequence of $\mathcal{X}_{\alpha, \beta}(T)$-valued random variables $\mathcal{Y}_n=(\hat{\rho}_n, \hat{w}^Q_n)$ and a $\mathcal{X}_{\alpha, \beta}(T)$-valued random variable $\mathcal{Y}=(\hat{\rho}, \hat{w}^Q)$, all defined on the probability space $(\hat{\Omega}, \hat{\mathcal{F}}, \hat{P})$, such that
\begin{equation}
\label{sm6}
\mathcal{L}(\mathcal{Y}_n)=\mathcal{L}(\rho^{\e_n}, w^Q),	
\end{equation}
and
\begin{equation}
	\label{sm7}
	\lim_{n\to\infty} \left(\Vert \hat{\rho}_n-\hat{\rho}\Vert_{L^\infty(0,T;H^{-\alpha})}+\Vert \hat{\rho}_n-\hat{\rho}\Vert_{L^{8/(3+\beta)}(0,T;H^{\beta})}+\Vert \hat{w}^Q_n-\hat{w}^Q\Vert_{C([0,T];U)}\right)=0,\ \ \ \ \ \hat{\mathbb{P}}-\text{a.s.}
\end{equation}

Now, we have
\[\begin{aligned}
\int_{0}^{t}&\Inner{\text{div}\left(b(\hat{\rho}_n(s))\nabla\hat{\rho}_n(s)\right),\varphi}_{H}ds =-\int_0^t\Inner{b(\hat{\rho}_n(s))\nabla\hat{\rho}_n(s),\nabla\varphi}_{H}ds\\[10pt]
&=-\int_0^t\Inner{\nabla(B(\hat{\rho}_n(s)),\nabla\varphi}_{H}ds=\int_0^t \Inner{B(\hat{\rho}_n(s)),\Delta\varphi}_{H}ds,
	\end{aligned}\]
	and thanks to \eqref{sm6} and \eqref{sm7}, this gives
 for every $\varphi\in C^{\infty}_{0}(\mathcal{O})$,
	\begin{equation*}
		\begin{array}{ll}
			\ds{\Inner{\hat{\rho}_{n}(t),\varphi}_{H} }
			&\ds{=\Inner{\mathfrak{r}_0,\varphi}_{H}+\int_{0}^{t}\Inner{B(\hat{\rho}_{n}(s)),\Delta\varphi}_{H}ds-\epsilon_{n}\int_{0}^{t}\Inner{\hat{\rho}_{n}(s),\Delta^{2}\varphi}_{H}ds }\\
			\vs
			&\ds{\quad +\int_{0}^{t}\Inner{F_{g}(\hat{\rho}_{n}(s)),\varphi}_{H}ds+\int_{0}^{t}\Inner{\varphi,\sigma_{g}(\hat{\rho}_{n}(s))d\hat{w}_{n}^{Q}(s)}_{H}. }
		\end{array}
	\end{equation*}
	Thus, by using the general argument introduced in \cite[proof of Theorem 4.1]{debussche16}, thanks to \eqref{sm7}  we can take the limit as $n\to\infty$   of both sides in the equality above, and we obtain that  $\hat{\rho}$ satisfies \eqref{weak_test_app}, with $w^Q$ replaced by $\hat{w}^Q$. Moreover, $\hat{\rho}\in L^{2}(\hat{\Omega};L^{\infty}(0,T;H))\cap L^{2}(0,T;H^{1}))$ and satisfies \eqref{approx_energy}, with $\mathbb{E}$ replaced by $\hat{\mathbb{E}}$.
		
	{\em Step 2.} We have that there exists a unique solution $\rho \in\,L^2(\Omega;C([0,T];H)\cap L^2(0,T;H^1))$ that satisfies \eqref{approx_energy}.	
	
	{\em Proof of Step 2.}
	Due to what we have seen above, there exists a unique solution 
	\[\rho \in\,L^2(\Omega;L^\infty(0,T;H)\cap L^2(0,T;H^1))\] that satisfies \eqref{approx_energy}. It only remains to prove that $\rho \in\,C([0,T];H)$, $\mathbb{P}$-a.s.
	 By proceeding as in \cite[Section 4.3]{debussche16} we consider the problem
		\[
		\le\{\begin{array}{l}
			\ds{\partial_{t}\xi(t,x)=\Delta\xi(t,x)+\sigma_{g}(\rho(t,\cdot))\partial_{t}w^{Q}(t,x), }\\[10pt]
			\ds{\xi(0,x)=\mathfrak{r}_0(x),\ \ \ \ \ \ \ \xi(t,\cdot)|_{\partial\mathcal{O}}=0, }
		\end{array}\r.
	\]
	whose unique solution $\xi$ belongs to  $L^2(\Omega;C([0,T];H)\cap L^2(0,T;H^1))$.
	Then, if we denote $\eta(t):=\rho(t)-\xi(t)$, we have that $\eta\in L^{\infty}(0,T;H)\cap L^{2}(0,T;H^{1})$, $\mathbb{P}$-a.s., and solves 
	\begin{equation}\label{limit_para_2}
		\le\{\begin{array}{l}
			\ds{\partial_{t}\eta(t,x)=\text{div}\big(b(\rho(t,x))\nabla\eta(t,x)\big)+\text{div}\big[(b(\rho(t,x))-I)\nabla\xi(t,x)\big]+f_{g}(x,\rho(t,x)), }\\[10pt]
			\ds{\eta(0,x)=0,\ \ \ \ \ \ \ \ \eta(t,\cdot)|_{\partial\mathcal{O}}=0. }
		\end{array}\r.
	\end{equation}
Now, if we denote by $U(t,s)$ the evolution family associated with the time-dependent differential operator
\[\mathcal{L}_t\varphi(x)=\text{div}\left[b(\rho(t,x))\nabla\varphi(x)\right],\ \ \ \ \ x \in\,\mathcal{O},\]
we have that
\[\eta(t,x)=	\int_0^t U(t,s)\left[\text{div}\left[b(\rho(t,\cdot))-I\right]\nabla \xi(s,\cdot)+f_g(\cdot,\rho(s,\cdot))\right](x)\,ds,\]
and since $\xi \in\,L^2(0,T;H^1)$ and $\rho \in\,L^\infty(0,T;H)$, $\mathbb{P}$-a.s., we get that $\eta \in\,C([0,T];H)$, $\mathbb{P}$-a.s. In particular, $\rho=\eta+\xi$ belongs to $C([0,T];H)$, $\mathbb{P}$-a.s.
	
\end{proof}

\subsection{Well-posedness of equation \eqref{limiting_problem} in $H$} 

From the well-posedness of the quasilinear stochastic parabolic equation \eqref{limit_para}, we get the well-poseness of equation \eqref{limiting_problem} in $H$.

By proceeding as in the proof of \cite[Theorem 7.1]{cerraixi}, we can show that  $u\in L^{2}(\Omega;C([0,T];H)\cap L^{2}([0,T];H^{1}))$ is a  solution to equation \eqref{limiting_problem} with initial condition $\mathfrak{u}_0\in L^{2}(\Omega;H)$ if and only if $\rho:=g(u)$ is a weak solution to equation \eqref{limit_para} with initial value $\mathfrak{r}_0=g(\mathfrak{u}_0)\in L^{2}(\Omega;H)$. Moreover, as a consequence of the Lipschitz continuity of $g$ and $g^{-1}$ on $\mathbb{R}$, we have the following result:
\begin{Proposition}
	 Assume Hypotheses \ref{Hypothesis1}, \ref{Hypothesis2} and condition \eqref{smapp1}. For every $T>0$ and every $\mathfrak{u}_0\in L^{2}(\Omega;H)$, there exists a unique weak solution $u\in L^{2}(\Omega;C([0,T];H)\cap L^{2}(0,T;H^{1}))$, to equation \eqref{limiting_problem} such that 
	\[
		\E\sup_{t\in[0,T]}\Vert u(t)\Vert _{H}^{2}+\E\int_{0}^{T}\Vert u(t)\Vert _{H^{1}}^{2}\,dt\leq c_T\left(1+\E\Vert \mathfrak{u}_0\Vert _{H}^{2}\right).
	\]
\end{Proposition}

%%%%%%%%%%%%%%%%%%%%%%%%%%%%%%%


\begin{thebibliography}{99}

%\bibitem{butkovsky} Butkovsky, O., Kulik, A. and Scheutzow, M., 2020. Generalized couplings and ergodic rates for SPDEs and other Markov models.

\bibitem{bhw} J.~Birrell, S.~Hottovy, G.~Volpe, J.~Wehr, {\em Small mass limit of a Langevin equation on a manifold}, Annales Henri Poincar\'e, Theoretical and Mathematical Physics, 18 (2017),  pp. 707--755.

\bibitem{BC23} Z.~Brzezniak, S.~Cerrai, {\em Stochastic wave equations with constraints: well-posedness and Smoluchowski-Kramers diffusion approximation}, arXiv: 2303.09717.

\bibitem{SK1} S.~Cerrai, M.~Freidlin, {\em On the Smoluchowski-Kramers approximation for a
system with an infinite number of degrees of freedom}, Probability
Theory and Related Fields 135 (2006), pp. 363-394.


\bibitem{SK2} S.~Cerrai, M.~Freidlin, {\em Smoluchowski-Kramers
approximation for a general class of SPDE's},  Journal
of Evolution Equations 6 (2006), pp. 657-689.

\bibitem{SK3} S.~Cerrai, M.~Freidlin, {\em Small mass asymptotics for a charged particle in a magnetic field and longtime
influence of small perturbations}, Journal of Statistical Physics 144 (2011), pp. 101--123.




\bibitem{cerraiglatt} S.~Cerrai, N.~Glatt-Holtz,  {\em On the convergence of stationary solutions in the Smoluchowski-Kramers approximation of infinite dimensional systems}, Journal of Functional Analysis 278, (2020), pp. 1--38.

\bibitem{sal} S.~Cerrai, M.~Salins, {\em Smoluchowski-Kramers approximation and large deviations for infinite dimensional gradient systems}, Asymptotic Analysis 88 (2013), pp. 201-215.

\bibitem{sal2}  S.~Cerrai, M.~Salins, {\em Smoluchowski-Kramers approximation and  large deviations for  infinite dimensional non-gradient systems with  applications to the exit problem},    Annals of Probability 44 (2016), pp. 2591--2642.


\bibitem{cs} S.~Cerrai, M.~Salins, {\em On the Smoluchowski-Kramers approximation for a
system with an infinite number of degrees of freedom subject to a magnetic field}, Stochastic Processes and their Applications 127 (2017) pp. 273--303.

\bibitem{CWZ} S.~Cerrai, J.~Wehr, Y.~Zhu, {\em An averaging approach to  the Smoluchowski-Kramers approximation in the presence of a varying magnetic field}, Journal of Statistical Physics 181 (2020), pp. 132--148.



\bibitem{cerraixi} S.~Cerrai, G.~Xi,  {\em A Smoluchowski-Kramers approximation for an infinite dimensional system with state-dependent damping},  Annals of Probability 50, (2022) pp.874--904.

\bibitem{cerraixie} S.~Cerrai, M.~Xie, {\em On the small noise limit in the Smoluchowski-Kramers approximation of nonlinear wave equations with variable friction}, arXiv preprint arXiv:2203.05923.

\bibitem{DP-Zab}
 G.~Da~Prato and J.~Zabczyk,
{\sc Stochastic equations in infinite dimensions}, Cambridge University Press, Second Edition, 2014.



\bibitem{debussche16} A.~Debussche, M.~Hofmanov\`a, J.~Vovelle, {\em   Degenerate parabolic stochastic partial differential equations: Quasilinear case}, Annals of Probability 44 (2016), pp. 1916--1955.

\bibitem{FoldesGlattHoltzRichardsThomann2013} J. ~F\"oldes,  N. ~Glatt-Holtz,  G. ~Richards,  E. ~Thomann, {\em Ergodic  and  mixing  properties  of  the  Boussinesq  equations  with a  degenerate  random  forcing}, Journal  of  Functional  Analysis  269 (2015) pp. 2427--2504.

\bibitem{f} M.~Freidlin, {\em Some remarks on the Smoluchowski-Kramers approximation},  J. Statist. Phys. 117 (2004), pp. 617--634.

\bibitem{fh} M.~Freidlin, W.~Hu, {\em Smoluchowski--Kramers approximation in the case of variable friction}, Journal of Mathematical Sciences 179 (2011), pp. 184--207.




\bibitem{frid} H.~Frid, Y.~Li,  D.~Marroquin, J.~F.~Nariyoshi,  Z.~Zeng, {\em The strong trace property and the Neumann problem for stochastic conservation laws}, Stochastics and Partial Differential Equations: Analysis and Computations 10 (2022), pp. 1--59.

\bibitem{fridnew} H.~Frid, Y.~Li,  D.~Marroquin, J.~F.~Nariyoshi,  Z.~Zeng, {\em The Dirichlet problem for stochastic degenerate parabolic-hyperbolic equations}, Communications in Mathematical Analysis and Applications 1 (2022), pp. 1--71.

\bibitem{GH} B.~Gess, M.~Hofmanov\`a, {\em Well-posedness and regularity for quasilinear degenerate parabolic-hyperbolic SPDE},  Annals of Probability 46 (2018), pp. 2495--2544.


\bibitem{HM} M.~Hairer, J.~Mattingly, {\em Spectral gaps in Wasserstein distances and the 2D stochastic Navier-Stokes equations}, Annals of Probability 36 (2008), pp. 2050--2091.

\bibitem{hhv} D.~Herzog, S.~Hottovy, G.~Volpe, {\em The small-mass limit for Langevin dynamics with unbounded coefficients and positive friction},
Journal of Statistical Physics 163 (2016), pp. 659--673.

\bibitem{hmdvw}S.~Hottovy, A.~McDaniel, G.~Volpe, J.~Wehr,
{\em The Smoluchowski-Kramers limit of stochastic differential equations with arbitrary state-dependent friction}, Communications in Mathematical Physics 336 (2015), pp. 1259--1283.

\bibitem{hu} W.~Hu, K.~Spiliopoulos, {\em Hypoelliptic multiscale Langevin diffusions:
large deviations, invariant measures and small mass asymptotics}, Electronic Journal of Probability 22 (2017).




\bibitem{kra} H.~Kramers, {\em Brownian motion in a field of
force and the diffusion model of chemical reactions}, Physica 7
(1940), pp. 284--304.

\bibitem{lee} J.~J.~Lee, {\em Small mass asymptotics of a charged particle in a variable magnetic
field}, Asymptotic Analysis 86 (2014), pp. 99--121.

\bibitem{Lv2} Y.~Lv, A.~Roberts, {\em Averaging  approximation  to  singularly  perturbed  nonlinear  stochastic  wave
equations}, Journal of Mathematical Physics 53 (2012), pp. 1--11.


\bibitem{Nguyen} H.~Nguyen, {\em The small-mass limit and white-noise limit of an infinite dimensional generalized Langevin equation}, Journal of Statistical Physics  173 (2018), pp. 411--437.


\bibitem{salins} M.~Salins, {\em Smoluchowski-Kramers approximation for the damped stochastic wave equation with multiplicative noise in any spatial dimension}, Stochastic  Partial Differential Equations: Analysis and Computation 7 (2019), pp. 86--122.



\bibitem{simon}
J.~Simon,
{\em Compact sets in the space {$L^p(0,T;B)$}}, 
Annali di Matematica Pura ed Applicata 146 (1986), pp. 65--96.

\bibitem{smolu} M.~Smoluchowski, {\em Drei Vortage \"uber
Diffusion Brownsche Bewegung und Koagulation von Kolloidteilchen},
Physik Zeit. 17 (1916), pp. 557-585.

\bibitem{spi}  K.~Spiliopoulos, {\em A note on the Smoluchowski-Kramers approximation for the Langevin equation with reflection}, Stochastics and Dynamics 7 (2007), pp. 141--152.


%\bibitem{hairermatt} Hairer, M., Mattingly, J.C. and Scheutzow, M., 2011. Asymptotic coupling and a general form of Harris’ theorem with applications to stochastic delay equations. Probability theory and related fields, 149, pp.223-259.

\end{thebibliography}
\end{document}